%% file: SFandASF_shortened_-_December_2019.tex
\documentclass[12pt,reqno]{amsart}
\usepackage{amsaddr}   
\usepackage{comment}
\usepackage{xspace}
\usepackage[colorlinks = true,
			citecolor = blue,
			linkcolor = blue]{hyperref}
\usepackage[]{todonotes}			  
\usepackage{amssymb}
\usepackage{physics}
\usepackage{mathtools}
\mathtoolsset{showonlyrefs}     
\usepackage{booktabs}
\usepackage[] {datetime2}
\usepackage{enumerate}

\author[C. Khor]{Calvin Khor}
\address{School of Mathematical Sciences, Beijing Normal University, \\Beijing 100875, P. R.  China}
\email{C.Khor@bnu.edu.cn}

\author[J.L. Rodrigo]{Jos\'e L. Rodrigo}
\address[J.L. Rodrigo]
{Mathematics Research Centre,
Zeeman Building, \\
University of Warwick,
Coventry CV4 7AL,\\
United Kingdom}
\email{J.Rodrigo@warwick.ac.uk}

\newcommand{\del}{\partial}
\renewcommand{\grad}{\nabla}
\newcommand{\br}[1]{\left(#1\right)}
\newcommand{\init}{\text{in}}
\newcommand{\out}{\text{out}}
\newcommand{\md}{\text{mid}}
\newcommand{\mynear}{\text{near}}
\newcommand{\myfar}{\text{far}}
\newcommand{\Torus}{\mathbb T}

\newcommand{\auxG}{\mathcal G}

\newcommand{\tgt}{T}
\newcommand{\nrml}{N} 
\DeclareMathOperator{\supp}{supp}

\newcommand{\deltgperp}{\mycurve_\tau \cdot \nrml}
\newcommand{\deltgpar}{\mycurve_\tau \cdot \tgt}

\newcommand{\mycurve}{z}
\newcommand{\myvel}{u}
\newcommand{\myspine}{\vec{S}}
\newcommand{\hlapl}{\Lambda}
    
\newcommand{\myspinein}{A_{\init}^{\myspine}}
\newcommand{\myspineout}{A_{\out}^{\myspine}}
\newcommand{\myspinemid}{A_{\md}^{\myspine}}

\newcommand{\indicator}{{\mathbf 1}}

 \DeclareRobustCommand{\SkipTocEntry}[5]{}

\newtheorem{thm}{Theorem}[section]
\newtheorem{cor}[thm]{Corollary}
\newtheorem{lem}[thm]{Lemma}
\newtheorem{prop}[thm]{Proposition}
\theoremstyle{definition}
\newtheorem{defn}[thm]{Definition}

\theoremstyle{remark}
\newtheorem{rem}[thm]{Remark}
\numberwithin{equation}{section}

\newcommand{\ASF}{ASF\xspace}
\newcommand{\set}[1]{\left\{{#1}\right\}}
\newcommand{\Real}{\mathbb R}
\newcommand{\eps}{\varepsilon}

\DeclareMathOperator{\mysin}{Sin}

\newcommand{\thetaU}[0]{ u}

\allowdisplaybreaks

\usepackage{cite}

\begin{document}

\title[On SFs and ASFs for singular SQG]{On Sharp Fronts and Almost-Sharp Fronts for singular SQG }
\date{\DTMnow}  
\maketitle 



\section*{Abstract}

In this paper we  consider a family of active scalars with a velocity field given by $\thetaU = \Lambda^{-1+\alpha}\grad^{\perp} \theta$, for $\alpha \in (0,1)$. This family of equations is a more singular version of the two-dimensional Surface Quasi-Geostrophic (SQG) equation, which would correspond to $\alpha=0$.

 We consider the evolution of sharp fronts by studying families of almost-sharp fronts. These are smooth solutions with simple geometry in which a sharp transition in the solution occurs in a tubular neighbourhood (of size $\delta$). We study their evolution and that of compatible curves, and introduce the notion of a spine for which we obtain improved evolution results, gaining a full power (of $\delta$) compared to other compatible curves.

\input{introforapproxeqnpaper}

\input{notation}
\input{evoforsharpfronts}

\input{almostsharpfronts}
\input{tubnbdcoordinates}
\input{fullequation}

\input{approxequation}
\input{removalofsingularitybyintegration}
\input{spinecompatiblecurveevolution}
\input{spineevolution}
\section{Acknowledgements}
Calvin Khor is supported by the studentship part of the ERC consolidator project $\text n^0$ $616797$. Jos\'e L. Rodrigo is partially supported by the ERC consolidator project  $\text n^0$ $616797$.
\appendix
\input{asymptoticintegral}

\end{document}

%% file: introforapproxeqnpaper.tex
\section{Introduction}

In this paper, we develop and study the notions of  sharp front, almost-sharp front, and spine curve for the following more singular version of the two-dimensional SQG equation,
\begin{equation}\label{eqn-singular-sqg} 
\left\{
 \begin{split}
        \partial_t \theta + u\cdot\grad \theta = 0, \\
        u = \grad^\perp \Lambda^{-1+\alpha} \theta  .
    \end{split}   \right.\end{equation}
We will focus in the range  $\alpha \in (0,1)$. Here $\theta=\theta(x,t)\in\mathbb R$, $x\in\mathbb R^2,$ and $t\ge 0$ and the operator $\Lambda = |\grad|$ is the half Laplacian on $\mathbb R^2$. To simplify the presentation we will define $\Lambda^{-1+\alpha}$ as given by a convolution with the kernel $K(x) = |x|^{-1-\alpha} \in L^2_{\text{loc}}(\mathbb R^2)$,
\[ \Lambda^{-1+\alpha} f(x) =  K * f(x) = \int_{\mathbb R^2} \frac{f(y)}{|x-y|^{1+\alpha}} \dd{y} .\]
In $\mathbb{R}^2$, usually one has $\Lambda^{-\beta} f = c_\beta |x|^{\beta - 2} * f$ but we ignore  $c_\beta$ here. For $\alpha\in[-1,0]$, the family interpolates between the 2D Euler equation($\alpha=-1$) and the 2D SQG equation ($\alpha=0$). The model with $\alpha\in(0,1)$ is the natural extrapolation of this family and the focus of this paper.

Sharp fronts are weak solutions of \eqref{eqn-singular-sqg} that have the special form
\[ \theta(x,t) = \mathbf 1_{x\in A(t)}, \]  where $A$ is a bounded simply connected domain with sufficiently  regular boundary (say $C^2$), and $\mathbf 1_U$ is the indicator function of the set $U$. 

Periodic graph-type SQG sharp fronts have been considered in \cite{rodrigo2004vortex} and \cite{fefferman2011analytic} by periodising one of the space variables. For SQG,  sharp fronts  have already been studied in the Sobolev setting in \cite{gancedo2008existence} and \cite{gancedo2016arXiv160506663C} (see  Chae et. al. \cite{chae2012generalized} 
 as well). 
 
 Non-periodic graph-type sharp fronts of \eqref{eqn-singular-sqg} were studied in the recent papers, and  of Hunter, Shu, and Zhang. They gave two different approaches in \cite{hunter2018regularized} and \cite{hunter2019contour} to deriving the correct contour dynamics equation in this setting, and proved \cite{hunter2018global} global existence under a certain smallness condition. They also study two-front dynamics \cite{hunter2019twofront} and approximate equations for sharp fronts \cite{hunter2018local}; these differ from the almost-sharp fronts considered in this paper, which are smooth functions solving \eqref{eqn-singular-sqg} exactly, approximating a sharp front. 

 Finally, a second paper \cite{khor2019analytic} by the authors has been prepared alongside this one that proves local existence of sharp fronts of \eqref{eqn-singular-sqg} a bounded domain in the analytic setting. This is an analogous result to the existence theorem for SQG in \cite{fefferman2011analytic}.

Analogously to \cite{fefferman2012almost}  and \cite{fefferman2015construct}, we define an almost-sharp front (ASF) of size $\delta$ as a regular\footnote{We do not attempt to optimise the regularity conditions for the front or almost-sharp front, taking them to be $C^2$ throughout.}  approximation of a sharp front: {
\[ \theta(x,t) = \begin{cases}
    1 & x\in A_{\init}(t), \\
    \text{regular} & x\in A_{\md}(t), \\
     0 & x \in \mathbb R^2 \setminus (A_{\init}(t)\cup A_{\md}(t)),
\end{cases} \]
where  $\supp \nabla\theta \subset A_{\md}$ and this `transition region' has area $O(\delta)$, so that in the limit as $\delta\to 0$, one formally recovers a sharp front.


In the Appendix, we derive asymptotics for integrals appearing in the equation, which we use to write down an approximate equation for an almost-sharp front. In doing so, we discover a significant difference with  \cite{fefferman2012almost}, which corresponds to $\alpha=0$: the appearance of a logarithm in the equation of \cite{fefferman2012almost} and the property $\log(ab) = \log b + \log b$ led to the development of an `approximate unwinding' for their equation in \cite{fefferman2015construct}. However, in our case, the singularity is stronger and we  have $|ab|^{-\alpha} \neq |a|^{-\alpha} + |b|^{-\alpha} $, so we have a nonlinear term that explodes as $\delta\to0$. Regardless, the analogous function to the $h$ function of \cite{fefferman2012almost} does satisfy a well behaved linear homogeneous equation in the limit as $\delta\to 0$.

The remainder of this paper is organised as follows. In Section \ref{section-notation}, we set up some notation. Section \ref{section3} discusses sharp fronts, with almost-sharp fronts introduced in Section \ref{section4}. The regularisation effect for $\Omega$ that leads to a limit equation for $h$ is discussed in Section \ref{section-h-equation}. Finally, Section \ref{section6} discusses the notion of compatible curves for almost-sharp fronts, and Section \ref{section-spine} introduces the notion of a spine, and establishes an improved evolution result.    

%% file: notation.tex
\section{Notation}
\label{section-notation}
Let $\mathbb T := \mathbb R / \mathbb Z$, with fundamental domain $[-1/2,1/2)$. The $90^\circ$ counter-clockwise rotation of $x = \binom{a}{b}\in\mathbb R^2 $ is $x^\perp := \binom{-b}{a},$
and similarly $\nabla^\perp=\binom{-\del_{x^2}}{\del_{x^1}}$.
We will frequently employ the shorthand notation
\begin{align}
    \label{shorthand}
        \int_A F(g,f_*) \dd{s_*} &:= \int_A F(g(s),f(s_*))\dd{s_*}, \\
    \int_A\int_B F(g,f_*)  \dd{s_*}\dd{\xi_*} &:= \int_A\int_B F(g(s,\xi), f(s_*,\xi_*)) \dd{s_*}\dd{\xi_*}.
\end{align}
That is, evaluation at $s,\xi$ is assumed unless a function is subscripted by $*$, and then we will assume it is evaluated at $s_*,\xi_*$.

%% file: evoforsharpfronts.tex
\section{Evolution equation for a sharp front}
\label{section3}
In this section, we give the contour dynamics equation (CDE) for sharp fronts. Recall that $K(x) = |x|^{-1-\alpha}$. The equation we study \eqref{eqn-singular-sqg} can be written equivalently as
\begin{equation}\label{eqn-general-evolution}
    \left\{ \begin{aligned}
        \del_t \theta + \myvel\cdot\grad \theta = 0, \\
        \myvel = (\grad^\perp K)* \theta .\\ 
    \end{aligned}\right.
    \end{equation}


\begin{defn}
    We say that $\theta=\theta(x,t)$ for $x\in\mathbb R^2,t\ge 0$ is a weak solution to \eqref{eqn-general-evolution} if there exists $T>0$ such that $\theta\in L^2(0,T; L^2(\mathbb R^2)) $,  $u=\nabla^\perp  K * \theta\in L^2(0,T; L^2(\mathbb R^2 ; \mathbb R^2)),$ and for any  $\phi\in C^\infty_c((0,T)\times \mathbb R^2)$
    \begin{align} \int_0^T\int_{\mathbb R^2} \theta\del_t\phi  + \myvel \cdot \nabla \phi \dd{x}\dd{t}= 0. \label{eqn-weak-solution-defn} \end{align}
\end{defn}

\begin{defn} 
(Sharp front)  We say that a weak solution $\theta$ to \eqref{eqn-general-evolution} is a sharp-front solution to \eqref{eqn-general-evolution} if
    \begin{enumerate}
        \item  for each $t\in[0,T]$, there exists a bounded simply connected closed set $A(t)$  with sufficiently regular boundary ($C^2$), parameterised (anti-clockwise) by  $z=z(s,t)$, with $s \in \mathbb T, t \in [0,T]$,
        \item $\theta$ is the indicator function of $A(t)$ for each $t\in(0,T)$,  \[ \theta(x,t) = \mathbf 1_{x\in A(t)} := \begin{cases}
 1 &     x \in A(t), \\ 0 & x \notin A(t).
 \end{cases} \]
     \end{enumerate} 
\end{defn}


It will be useful to have coordinates defined on a fixed domain $\mathbb T$. Therefore,  we will not use arc-length, as the length of the curve may not be preserved in time. Instead, we  use a uniform speed parameterisation (also see Section \ref{tub-coords}) 
 $\mycurve:\mathbb T \to \mathbb R^2$ where the speed $\left|\dv{\mycurve}{s}\right| = L$ is constant in $s$, and $L(t)$ is the length of the curve at time $t$. If  $\mycurve$ is positively oriented, the 2D Frenet formulas  are
$\dv{\tgt}{s} = L \kappa \nrml,\dv{\nrml}{s} = -L \kappa \tgt$, with  unit tangent  $\tgt = \frac1L \dv{\mycurve}{s} $ and unit normal $\nrml := \tgt^\perp$. The curvature is $\kappa = \frac1L\dv{\tgt}{s}\cdot \nrml = L^{-3} \mycurve_{ss}\cdot\mycurve_s^\perp$. Integration along $\mycurve$ will be written as 
$$ \int_{\mycurve(\mathbb T) } f(\mycurve) \dd{l} := L \int_{\mathbb T} f(z(s)) \dd{s}.$$
A similar argument to \cite{rodrigo2004vortex} yields a rigorous derivation of the evolution equation for $z$.
\begin{prop}[Evolution of a sharp front]\label{prop-evo-of-SF}If $\theta=\indicator_A$ is a sharp front solution to \eqref{eqn-general-evolution}, then the uniform speed counter-clockwise parameterisation $\mycurve:\mathbb T\to\mathbb R^2$ of $\del A$ with  normal $\nrml = \del_s \mycurve^\perp$ satisfies the following CDE,
\begin{equation}
	\label{eqn-generic-cde}
	 \del_t \mycurve \cdot \nrml = \left(- \int_{\mathbb T} K(\mycurve-\mycurve_*) (\del_s \mycurve_* - \del_s \mycurve )  \dd{s}_*  \right) 
\cdot \nrml =: -\mathcal I(\mycurve) \cdot \nrml.
\end{equation}
\end{prop}

%% file: almostsharpfronts.tex
\section{Almost-Sharp Fronts}
\label{section4}
In this chapter, we define an almost-sharp front as a regularisation (in a small strip) of a sharp front. Then we  derive an asymptotic equation for an almost-sharp front, using an asymptotic lemma for a parameterised family of integrals. We also show that integrating across the transition region has a regularising effect, simplifying the asymptotic equation of an almost-sharp front.

\begin{defn}[Almost-sharp front]\label{defn-ASF}
An almost-sharp front (\ASF)  $\theta=\theta_\delta(x,t)$ to SQG is a family of solutions to \eqref{eqn-general-evolution}, such that for each $\delta > 0$ sufficiently small, there exists a closed $C^2$ curve  $\mycurve$ bounding a simply connected region $A$, and a constant $C^{\mycurve}$ with $0 <C^{\mycurve} < C$ for some fixed $C$ independent of $\delta$, such that with the following sets,
\begin{alignat*}{2}
 A_{\md}(t)& 
    &&= \set{x \in \mathbb R^2 : \operatorname{dist}(x,\partial A(t)) \le  C^{\mycurve}  \delta },\\
 A_{\init}(t) &
    &&= \set{x \in A : \operatorname{dist}(x,A(t)) > C^{\mycurve}  \delta},\\
 A_{\out}(t)  &
    &&= \Real^2 \setminus (A_{\init} \cup A_{\md}),    
\end{alignat*}
$\theta$ is a $C^2$ function in $\mathbb R^2$ (as a function of $x$) such that
	\[\theta(x,t) = \begin{cases}
	1 & x \in A_{\init}(t), \\
	C^2 \text{ smooth} & x \in A_{\md}(t),\\
	0  & x \in A_{\out}(t).
\end{cases}
\]
We also demand the following growth condition for $\theta_\delta$,
\begin{align}\label{eqn-growth-condition}
    \|\nabla \theta \|_{L^\infty} \lesssim \frac{1}{\delta}.
\end{align}
\end{defn}
In particular $\theta$ is locally constant  in $A_\out(t)\cup A_\init(t)$,  and
	\[ \supp \theta \subset A_\md(t)\cup A_{\init}(t),\quad  \supp \grad \theta(\cdot,t)\subset A_{\md} (t).\]
Note that the curve $\mycurve$ in the above definition is not unique.
\begin{defn}[Compatible curve]\label{defn-comp-curve}
Any curve $\mycurve$ satisfying the above definition for an ASF $\theta$ is called a compatible curve for $\theta$.
\end{defn}

%% file: tubnbdcoordinates.tex
\subsection{Tubular neighbourhood coordinates}
\label{tub-coords}
Given a curve $\mycurve(s,t)$ with uniform speed (denoted by $L(t)$),  we  define the tubular neighbourhood coordinates around $\mycurve$ (for $\delta\ll 1$) by (we suppress the explicit dependence on $t$)
\begin{align}\label{eqn-tub-coords}  x(s,\xi) =  \mycurve(s) + \delta \xi \nrml(s), \ s\in\mathbb T, \xi \in [-1,1]. \end{align}
We will denote by $L_1 := L (1-\delta \kappa\xi)$, where as before $\kappa$ is the curvature of $z$, and $L(t)$ is the length of the curve (which agrees with the uniform speed of the parameterisation).
With  $s,\xi$ as above and setting $\tau = t$, 
then elementary calculations yield 
\[
\del_{x^1} =  \dfrac{\tgt^1 }{L_1} \del_s -\dfrac{\tgt^2}{\delta} \del_\xi, \qquad \
\del_{x^2} =  \dfrac{\tgt^2 }{L_1} \del_s +\dfrac{\tgt^1}{\delta} \del_\xi, 
\qquad \
\del_t     =   A_0^1\del_s  + A_0^2 \del_\xi  +  \del_\tau,
\]where $A_0 = \binom{A_0^1}{A_0^2} = \binom{ -x_\tau \cdot \tgt/L_1 }{ -x_\tau \cdot \nrml/\delta 
} $, 
and $\dd{x}$ under the change of variables is $\dd{x^1}\dd{x^2} = \left|\det \dv{(x^1,x^2,t)}{(s,\xi,\tau)} \right|\dd{s}\dd{\xi} = L_1 \delta \dd{s}\dd{\xi}$.

%% file: fullequation.tex
\subsection{Equation in tubular neighbourhood coordinates}

Define $\Omega$ as $\theta$ expressed in the  tubular neighbourhood coordinates, $\Omega(s,\xi,\tau) = \theta(x,t)$. Then the gradient of $\theta$ and its perpendicular can be written as
\[  \grad\theta = \tgt \frac{\Omega_s}{L_1 }  + \nrml \frac{\Omega_\xi}{\delta} ,
 \qquad
\grad^\perp\theta = \nrml \frac{\Omega_s}{L_1 }  -\tgt  \frac{\Omega_\xi}{\delta} . \]

Since $ \myvel(x,t) = \int_{\mathbb R^2} \frac{\nabla^\perp\theta(x_*) }{|x-x_*|^{1+\alpha}} \dd{x_*}$ from \eqref{eqn-singular-sqg}  we obtain the following expression for $u$ in terms of $\Omega$ (using  $x=x(s,\xi,\tau)$ and $x_*=x(s_*,\xi_*,\tau)$),
\begin{align}
\myvel(x) & = \iint_{\mathbb T \times [-1,1]} \frac{\delta \nrml_*\Omega_{s*} - L_{1*}\tgt_*\Omega_{\xi*} }{|x-x_*|^{1+\alpha} } \dd{\xi_*} \dd{s_*}.
\end{align}
Then since
\begin{align*}  \nabla \theta(x) \cdot \nabla^\perp\theta(x_*)
&=\nrml_* \cdot\tgt \frac{\Omega_{s*} \Omega_s }{L_{1*}L_1} 
+ \nrml_*\cdot\nrml \frac{\Omega_{s*} \Omega_\xi  }{\delta L_{1*}} 
\\ 
&\quad- \tgt_* \cdot \tgt \frac{ \Omega_{\xi*} \Omega_s}{\delta L_1} 
- \tgt_* \cdot \nrml \frac{\Omega_{\xi*}  \Omega_\xi }{\delta^2},
\end{align*}
we can write the $\myvel \cdot \nabla \theta$ term in the equation \eqref{eqn-general-evolution} as
\begin{align*}
    \myvel\cdot\grad\theta(x) 
 &=
  \iint_{\mathbb T\times [-1,1]} \frac{1}{|x - x_*|^{1+\alpha}} 
\big(
\nrml_* \cdot\tgt \frac{\delta \Omega_{s*} \Omega_s }{L_1} 
+ \nrml_*\cdot\nrml \Omega_{s*} \Omega_\xi   
\\ 
&\quad- \tgt_* \cdot \tgt \frac{L_{1*} \Omega_{\xi*} \Omega_s}{ L_1} 
- \tgt_* \cdot \nrml \frac{L_{1*} \Omega_{\xi*}  \Omega_\xi }{\delta}  
 \big)
  \dd{\xi}_* \dd{s}_*. \tag{\theequation}\label{u-grad-theta-in-new-coords} 
\end{align*}
Therefore, using $\frac{L_{1*}}{L_1}  = 1 - \frac{\delta(\xi_* \kappa_* - \xi \kappa)}{1-\delta \kappa \xi}$, the equation \eqref{eqn-general-evolution} can be written in the new coordinates as  (using $\nabla^\perp \Omega_* \cdot \nabla\Omega = \Omega_{s*}\Omega_\xi  - \Omega_{\xi*}\Omega_s$)
\begin{align*}
        0 &= \Omega_\tau +
          \Omega_s
        \left(
            A_0^1 
             +\iint_{\mathbb T\times[-1,1]}  \frac{\delta \nrml_*\cdot\tgt}{L_1 |x_*-x|^{1+\alpha}}\Omega_{s*}  \dd{\xi}_* \dd{s}_* \right.\\ 
        &\ \ \left.+ \iint_{\mathbb T\times[-1,1]} \frac{ \delta (\xi_*\kappa_* - \xi\kappa) \tgt_* \cdot \tgt}{(1-\delta\xi\kappa)|x_* - x|^{1+\alpha}} \Omega_{\xi*}   \dd{\xi}_* \dd{s}_*
            \right) \\
        &\ \  + \left( \iint_{\mathbb T\times[-1,1]}  \frac{  \tgt_* \cdot \tgt}{|x_* - x|^{1+\alpha}} \grad^\perp\Omega_*     \dd{\xi}_* \dd{s}_* \right) \cdot \grad\Omega  \\
        &\ \    + \frac{1}{\delta}\Omega_{\xi}\br{\iint_{\mathbb T\times[-1,1]} \frac{-L_{1*}\tgt_*\cdot\nrml}{|x_*-x|^{1+\alpha} }  \Omega_{\xi*}   \dd{\xi}_* \dd{s}_* - \deltgperp}. 
				\label{eqn-sqg-v2a1} \tag{\theequation}  
\end{align*}

%% file: approxequation.tex
\subsection{An approximate equation for an almost-sharp front}\label{section-approx-eqn}
We rewrite \eqref{eqn-sqg-v2a1} as (this defines the terms $I_i$) 
\begin{align}\label{eqn-symbolic}
0 =  \Omega_{\tau} +\Omega_s(A^1_0 + \delta I_1 + \delta I_2) + I_3\cdot\grad \Omega + \frac{\Omega_\xi}\delta (I_4 - \deltgperp), \end{align}
where $A^1_0 = \frac{-\mycurve_\tau \cdot \tgt}L+ O(\delta)$, and aside from explicit dependencies,
$I_i$ depends on $\delta$ through the coordinate function $x=\mycurve+\delta\xi \nrml $ in \eqref{eqn-tub-coords}. 

We now want to find the leading order behaviour in $\delta\ll 1$. Asymptotic expansions for the integrals in \eqref{eqn-sqg-v2a1} will be obtained using Lemma \ref{prop-asymptotic-integral} and Corollary \ref{cor-to-asymptotic-lemma}, which are deferred  to Appendix \ref{sectionAppendix}. 
\begin{thm}[Approximate Equation for an \ASF]\label{thm-approx-eqn} Given $\Omega$, a $\delta$-ASF solution of the generalised SQG equation (in the sense of Definition \ref{defn-ASF},   tubular neighbourhood of the $\mycurve$) we have
\begin{align}
 & o(1) 
= \del_\tau \Omega - \frac{\deltgpar }{L} \del_s\Omega 
 \notag \\
& \,\,\,  +(2+2\alpha)L\int_{-1}^1\int_{\mathbb T} \frac{\tgt_*\cdot \nrml }{|\mycurve-\mycurve_*|^{3+\alpha}}   (\mycurve-\mycurve_*)\cdot(\xi\nrml - \xi_*\nrml_*)\del_\xi\Omega_* \dd{s}_* \dd{\xi}_* \del_\xi\Omega 
\\ 
& \,\,\,+  
	     \int_{-1}^1\int_{\mathbb T}
	     \frac{  L\kappa_* \tgt_*\cdot \nrml  }
	       {|
	       \mycurve - \mycurve_*|^{1+\alpha}} \xi_* \del_{\xi}\Omega_{*}      \dd{s_*}  \dd{\xi}_* \del_{\xi}\Omega  \notag  \\
& \,\,\, +\frac{C_{1,\alpha}\delta^{-\alpha}}{L}\int_{-1}^1 \frac{\grad\Omega_*^\perp|_{s_*=s}}{|\xi-\xi_*|^{\alpha}}\dd{\xi}_*  \cdot \grad\Omega \notag
			+\frac{C_{2,\alpha}}{L^{1+\alpha}}\int_{-1}^1 \grad\Omega^\perp|_{s_*=s}  \dd{\xi}_* \cdot \grad\Omega 
			\\ 
	     &\,\,\,
			+ \int_{-1}^1\int_\mathbb T 
        \frac{
            \tgt_*\cdot\tgt \nabla^\perp \Omega_*\cdot\nabla\Omega
            }{
            |\mycurve - \mycurve_*|^{(1+\alpha)/2}
            } 
        - 
        \frac{
            \nabla^\perp \Omega_*|_{s_*=s}\cdot\nabla\Omega
            }{
            L^{1+\alpha}|\mysin(s-s_*)|^{1+\alpha}
            } 
        \dd{s_*}\dd{\xi_*} \tag{\theequation} \label{eqn-full-approx-omega}.
\end{align}
\end{thm}
\begin{rem} In \eqref{eqn-full-approx-omega}, the integral term
\[  (2+2\alpha)L\int_{-1}^1\int_{\mathbb T} \frac{\tgt_*\cdot \nrml }{|\mycurve-\mycurve_*|^{3+\alpha}}   (\mycurve-\mycurve_*)\cdot(\xi\nrml - \xi_*\nrml_*)\del_\xi\Omega_* \dd{s}_* \dd{\xi}_* \del_\xi\Omega \]    
is well defined, despite the strength of the
 singularity in the denominator, since  the dot product introduces a  cancellation as long as say, $\mycurve \in C^3(\mathbb T)$. 
\end{rem}

\begin{proof}[Proof of Theorem \ref{thm-approx-eqn}]

To use the lemma, we will perform an approximation of the integral in $s_*$ for fixed $\xi,\xi_*,$ and $s$.
For given $s,\xi,\xi_*$, when applying Lemma \ref{prop-asymptotic-integral},  we will write $a_* = a(s_*) = a(s,s_*,\xi,\xi_*),$ and $a=a(s) = a(s,s,\xi,\xi_*)$.

We  rewrite  $I_1,I_2,I_3,I_4$ in \eqref{eqn-symbolic} in a form suitable for Lemma \ref{prop-asymptotic-integral}. We have   $x = \mycurve + \xi\delta \nrml$,  and $
    x_* = \mycurve + \xi_*\delta\nrml_* $ and so
\begin{align*}
    x_{s*} = \partial_{s_*} x_* &= L(1-\xi_*\delta\kappa_* )T_*,\\
    \tau^2 &= \delta^2(\xi_*-\xi)^2,\\
    g(s_*) &= |x_*-x|^2 - \tau^2,\\
    g'(s_*) &= 2x_{s*} \cdot x_* -2 x_{s*} \cdot x = 2 x_{s*} \cdot (x_* - x),\\
    g''(s_*) &= 2 | x_{s*}|^2 + 2  x_{ss*} \cdot (x_*-x) \\ &= 2L^2(1-\xi_*\delta\kappa_*)^2+2x_{ss*}\cdot (x_*-x).
\end{align*}
Since (as a function of $s_*$ with $s,\xi,\xi_*$ fixed) $g'(s)=0$ and $g''(s)>0$ for $\delta\ll 1$, $g$ has a non-degenerate minimum at $s_*=s$, with $g(s)=\tau^2-\tau^2 = 0$. As the curve $\mycurve$ has no self-intersections, this is the unique minimum. Therefore we have (as in  Lemma \ref{prop-asymptotic-integral})
\begin{equation} \label{eqn-definitions-to-apply-asymp-lem}
G(\delta) = \sqrt{\frac12 g''(s)} = L(1-\delta\kappa\xi_*).
\end{equation}
(Note that $G\neq L_1$, defined earlier as $L(1-\delta\kappa\xi)$.)
 For  $a_*=a(s_*)$, Lemma \ref{prop-asymptotic-integral} gives the following expansion for $I = \int_{\mathbb T} \frac{a_*}{|g_* + \tau^2 |^{(1+\alpha)/2}} \dd{s}_*$,  
\begin{align}
    I 
    &= \frac{a(s)}{G} C_{1,\alpha} \tau^{-\alpha} 
    + \frac{a(s)C_{2,\alpha}}{ G^{1+\alpha} } \label{eqn-asymplemma-appl1} \\
    &+ \int_{\Torus} \frac{a_*}{|g_*|^{(1+\alpha)/2}} 
        - \frac{a(s)}{G^{1+\alpha}|\mysin(s_*-s)|^{1+\alpha}} \dd{s}_*   
    + O(\tau^{2-\alpha}),
    \quad  \tau\to 0,  \label{eqn-asymplemma-appl2}
\end{align}  
where $\mysin$ is the rescaled $\mysin(s) := \sin(\pi s)/\pi$, and $C_{1,\alpha}, C_{2,\alpha}$ are two known constants. Also, with $g_*=g(s_*,\delta),G=G(\delta)$,
\begin{align}    \partial_\delta g(s_*) 
&= 2\del_\delta (x_*-x)\cdot (x_*-x) - 2\delta (\xi_* - \xi)\\
& =2(\xi_*\nrml_* - \xi\nrml)\cdot(\mycurve_*-\mycurve)+ 2\delta|\xi_*\nrml_* - \xi\nrml|^2 - 2\delta|\xi_*-\xi|^2,\\
\partial_\delta G &= - L \kappa \xi_* ,
\end{align}
so that
\begin{align} \del_\delta g(s_*,0) &=  2(\xi_*\nrml_* - \xi\nrml)\cdot(\mycurve_*-\mycurve),
\\
\del_\delta (G^{-1-\alpha})(\delta) &= (-1-\alpha)G^{-2-\alpha}\del_\delta G = \frac{(1+\alpha)}{L^{1+\alpha}(1-\xi_*\delta\kappa)^{2+\alpha}}\kappa \xi_* . \end{align}
 Corollary \ref{cor-to-asymptotic-lemma} in this setting gives
\begin{align*}
    I 
    &= \frac{a(s)}{L} C_{1,\alpha} \delta^{-\alpha}|\xi_*-\xi|^{-\alpha} 
    + \frac{a(s)C_{2,\alpha}}{ L^{1+\alpha} } 
    \\
    &+ \int_{\Torus} \frac{a(s_*)}{|g(s_*,0)|^{(1+\alpha)/2}} 
        - \frac{a(s)}{L^{1+\alpha}|\mysin(s_*-s)|^{1+\alpha}} \dd{s}_*  \\
    &+ \delta \int_{\Torus} \frac{(-2-2\alpha)a(s)}{|g(s_*,0)|^{(3+\alpha)/2}} (\xi_*\nrml_* - \xi\nrml)\cdot(\mycurve_*-\mycurve) - \frac{\del_\delta (G^{-1-\alpha})(0)a(s)}{|\mysin s|^{1+\alpha}} \dd{s_*} \\
    &+ O(\delta^{1-\alpha}).
\end{align*}  
The terms $I_1,I_2$ are simpler.
For $I_1$, $a(s_*) = \frac{\nrml_* \cdot \tgt}{L(1-\delta\kappa \xi)} \Omega_{s*}$, with $a(s) = 0$. Lemma \ref{prop-asymptotic-integral} gives
\[ \delta I_1 = \delta \int_{-1}^1\int_{\mathbb T} \frac{a_*}{|g_*|^{(1+\alpha)/2}} \dd{s_*}\dd{\xi_*} + O(\delta^{2-\alpha}) = O(\delta).
\]
For $I_2$, $a(s_*) = \frac{(\xi_*\kappa_* - \xi\kappa)\tgt_* \cdot \tgt }{1-\delta \xi \kappa}\Omega_{\xi*}$. Thus, Lemma \ref{prop-asymptotic-integral} gives
\begin{align*} &\delta I_2 = \delta^{1-\alpha} \int_{-1}^1  |\xi-\xi_*|^{-\alpha}\frac{ a}{L} C_{1,\alpha} \dd{\xi_*} + \delta \frac{a C_{2,\alpha}}{L^{1+\alpha}} 
\\ 
&+ \delta \int_{\mathbb T}  \frac{a_*}{|g_*|^{(1+\alpha)/2}} 
        - \frac{\pi^{1+\alpha}a }{L^{1+\alpha}|\sin(\pi (s_*-s))|^{1+\alpha}}  \dd{s}_*  + O(\delta^{1-\alpha}) \\
        &= O(\delta^{1-\alpha}).
\end{align*}
For $I_4$, since $L_{1*} = L(1-\delta \kappa_* \xi_*)$, we write $I_4 = I_{4,1} + \delta I_{4,2}$, with
\begin{align*}
    I_{4,1} &= \int_{-1}^1\int_{\mathbb T} \frac{-L\tgt_*\cdot\nrml}{|x_*-x|^{1+\alpha} }  \Omega_{\xi*} \  \dd{\xi}_* \dd{s}_*, \\ 
    I_{4,2} &= \int_{-1}^1\int_{\mathbb T} \frac{L\kappa_*\xi_*\tgt_*\cdot\nrml}{|x_*-x|^{1+\alpha} }  \Omega_{\xi*} \  \dd{\xi}_* \dd{s}_*. 
\end{align*}
For $I_{4,1}$, $a(s_*)= -L \tgt_* \cdot \nrml \Omega_{\xi*}$, and $a(s) = 0$.   Therefore, Lemma \ref{prop-asymptotic-integral} 
and Corollary \ref{cor-to-asymptotic-lemma} gives (since $\del_\delta g(s_*,0) = L(x-x_*)\cdot (\xi\nrml  - \xi_*\nrml_*)$)
\begin{align*} I_{4,1}     
&=  -L\int_{\mathbb T} \frac{  \tgt_* \cdot\nrml  \int_{-1}^1\Omega_{\xi*}\dd{\xi_*}}{|z-z_*|^{1+\alpha}} \dd{s_*} 
\\
& + (1+\alpha) \delta \int_{-1}^1\int_{\mathbb T}  
	     \frac{2 L  \tgt_* \cdot\nrml  }{|\mycurve-\mycurve_*|^{3+\alpha}}   (\mycurve-\mycurve_*)\cdot(\xi\nrml - \xi_*\nrml_*) \nonumber \del_\xi\Omega_* \dd{s}_* \dd{\xi}_* \\ &+ O(\delta^{2-\alpha}) .
\end{align*}
Since we have chosen $\int_{-1}^1\Omega_{\xi*}\dd{\xi_*} = 1$, the first term is equal to the sharp front evolution term $-\mathcal I (\mycurve) \cdot\nrml = \mycurve_\tau \cdot \nrml$, and 
\begin{align}
     &\frac1{\delta}(I_{4,1} - \mycurve_\tau \cdot \nrml  )
     \\ &= (1+\alpha)\int_{-1}^1\int_{\mathbb T}  
	     \frac{2 L  \tgt_* \cdot\nrml  }{|\mycurve-\mycurve_*|^{3+\alpha}}   (\mycurve-\mycurve_*)\cdot(\xi\nrml - \xi_*\nrml_*) \del_\xi\Omega_* \dd{s}_* \dd{\xi}_* + o(1).  
\end{align}
For $I_{4,2}$, $a(s_*) = L\kappa_*\xi_*  \tgt_* \cdot \nrml  \Omega_{\xi*}$. Lemma \ref{prop-asymptotic-integral} and Corollary \ref{cor-to-asymptotic-lemma} give
 \begin{align*}
 I_{4,2} & = \int_{-1}^1\int_{\mathbb T} \frac{L\kappa_*\xi_*\tgt_*\cdot\nrml \Omega_{\xi*}} {|x_*-x|^{1+\alpha} }     \dd{\xi}_* \dd{s}_* 
\\
 & 
= \int_{\mathbb T} \frac{L\kappa_* \tgt_*\cdot\nrml \int_{-1}^1 \xi_* \Omega_{\xi*}\dd{\xi}_* }{|\mycurve_*-\mycurve|^{1+\alpha} }     \dd{s}_*  + O(\delta).
 \end{align*}
 For $I_3$, $a(s_*) = \tgt_*\cdot \tgt \nabla^\perp \Omega_*$. Applying Lemma \ref{prop-asymptotic-integral} gives
 \begin{align*}
     I_3 
     &= \int_{-1}^1\Bigg( \frac{C_{1,\alpha}}{L(1-\delta\kappa \xi_*) \delta^\alpha}  \frac{\nabla^\perp \Omega_*|_{s_*=s}}{|\xi-\xi_*|^\alpha}
     + \frac{C_{2,\alpha} }{(L(1-\delta\kappa \xi_*))^{1+\alpha}} \nabla^\perp \Omega_*|_{s_*=s}
     \\
     & + \int_{\mathbb T} 
        \frac{
            \tgt_*\cdot\tgt \nabla^\perp \Omega_*
            }{
            |g_*|^{(1+\alpha)/2}
            } 
        - 
        \frac{  
            \pi^{1+\alpha}\nabla^\perp \Omega_*|_{s_*=s}
            }{
            (L(1-\delta\kappa \xi_*))^{1+\alpha}|\sin(\pi(s-s_*))|^{1+\alpha}
            } 
        \dd{s_*} \Bigg) \dd{\xi_*} \\
        &+ O(\delta^{1-\alpha}).
 \end{align*} 
 The remaining integral in $s_*$ is dealt with by applying Corrolary \ref{cor-to-asymptotic-lemma},
 \begin{align*}
     I_3
     &= \frac{C_{1,\alpha}}{L \delta^\alpha} \int_{-1}^1 \frac{\nabla^\perp \Omega_*|_{s_*=s}}{|\xi-\xi_*|^\alpha} \dd{\xi_*} 
     + \frac{C_{2,\alpha} }{L^{1+\alpha}} \int_{-1}^1 \nabla^\perp \Omega_*|_{s_*=s}\dd{\xi_*}
     \\
     &  + \int_\mathbb T 
        \frac{
            \tgt_*\cdot\tgt \int_{-1}^1\nabla^\perp \Omega_*\dd{\xi_*}
            }{
            |\mycurve - \mycurve_*|^{1+\alpha}
            } 
        - 
        \frac{
            \pi^{1+\alpha}\int_{-1}^1\nabla^\perp \Omega_*|_{s_*=s}\dd{\xi_*}
            }{
            L^{1+\alpha}|\sin(\pi(s-s_*))|^{1+\alpha}
            } 
        \dd{s_*}     + O(\delta^{1-\alpha}),
 \end{align*}
 completing the proof.
 \end{proof}
 
%

 

%% file: removalofsingularitybyintegration.tex
\section{A regularisation by integration across the Almost-Sharp Front}
\label{section-h-equation}
As noticed in \cite{fefferman2012almost}, the most singular terms in $\grad\Omega^\perp \cdot \grad\Omega_*|_{s_*=s} $  in the  equation disappear upon integration in $\xi$ over $[-1,1]$, since for any even integrable function $F$ and two functions $f,g$, we have the cancellation 
\[ \iint_{\xi,\xi_* \in [-1,1]} F(\xi - \xi_*)\big[ f(\xi)g(\xi_*) - f(\xi_*) g(\xi) \big] \ \dd{\xi}\dd{\xi}_* = 0. \]
The integrals with $\grad\Omega^\perp \cdot \grad\Omega_*|_{s_*=s} $ are of this form, with $f=\del_s\Omega$ and $g=\del_\xi\Omega$, for  fixed  $s$. In particular, the term of order $\delta^{-\alpha}$ can be written in this form. 
 This motivates shifting $\Omega$  by a constant so that  
\begin{equation}
    \Omega(s,\xi,t) = \begin{cases}
        -\frac12 & \xi \leq -1,\\
        C^2\text{ smooth} &\xi \in [-1,1],\\
        \frac12 & \xi \ge 1,
    \end{cases}\label{eqn-omega-pmonehalf}
\end{equation}
and we make the definition
\[h(s) := \int_{-1}^1 \Omega  \dd{\xi}. \]
The identities $
        \int_{-1}^1 \xi \del_\xi \Omega(s,\xi)   \dd{\xi}
        = -h(s)$, and
$        \int_{-1}^1 \del_\xi \Omega(s,\xi)  \dd{\xi} = 1$ 
yield
\begin{align*}
    o(1) 
     &= \del_\tau h -  \frac{\deltgpar }{L} h' 
     \\ &+(2+2\alpha)L \int_{s_*} \frac{\tgt_*\cdot \nrml }{|\mycurve-\mycurve_*|^{3+\alpha}}   (\mycurve-\mycurve_*)\cdot(h \nrml - h_*\nrml_*)   \dd{s}_* \\
     & +  
	     \int_{\mathbb T}
	     \frac{  L\kappa_* \tgt_* \cdot \nrml}
	       {|
	       \mycurve - \mycurve_*|^{1+\alpha}}   h_*    \dd{s_*}   + \int_{\mathbb T}
            \frac{\tgt_* \cdot \tgt(h'_*-h')}{|\mycurve - \mycurve_*|^{1+\alpha}} 
          \dd{s_*}  .
\end{align*}
Namely, in the limit $\delta \to 0$, $h$ evolves via a linear homogenous integrodifferential equation that doesn't depend on $\Omega$. Rewriting the $\xi_*$ integrals in \eqref{eqn-full-approx-omega} using the two further identities
\[\int_{-1}^1 (\xi \nrml - \xi_* \nrml_*) \del_\xi \Omega_* \dd{\xi} = (\xi\nrml + h_* \nrml_*), \quad \int_{-1}^1 \grad\Omega^\perp|_{s_*=s} \dd{\xi_*} = \begin{pmatrix}
    -1 \\ h'
\end{pmatrix},  \] 
we can treat $h$ as an independently evolving function coupled with $\Omega$ via the following equation, 
\begin{equation}
    \begin{split}
        o(1) &= \del_\tau \Omega - \frac{\mycurve_\tau \cdot T }{L} \del_s\Omega \\
& \quad + (2+2\alpha)L\int_{\mathbb T} \frac{\tgt_*\cdot \nrml }{|\mycurve-\mycurve_*|^{3+\alpha}}   (\mycurve-\mycurve_*)\cdot(\xi \nrml + h_* \nrml_*) \dd{s}_*  \del_\xi\Omega \\ 
&\quad +  
	     \int_{\mathbb T}
	     \frac{  L\kappa_* \tgt_* \cdot \nrml}
	       {|
	       \mycurve - \mycurve_*|^{1+\alpha}}   h_*    \dd{s_*} \del_\xi\Omega \\
& \quad +\frac{C_{1,\alpha}\delta^{-\alpha}}{L}\int_{-1}^1 \frac{\grad\Omega_*^\perp|_{s_*=s}}{|\xi-\xi_*|^{\alpha}}\dd{\xi}_*  \cdot \grad\Omega 
			+\frac{C_{2,\alpha}}{L^{1+\alpha}} \binom{-1}{h'} \cdot \grad\Omega \\ 
	     &\quad + \int_{\mathbb T}\br{\frac{\tgt_* \cdot \tgt\binom{-1}{h'_*}}{|\mycurve - \mycurve_*|^{1+\alpha}} - \frac{\binom{-1}{h'}}{L^{1-\alpha}|\mysin(s-s_*)|^{1+\alpha}}}  \dd{s}_* \cdot\grad\Omega.
    \end{split}
\end{equation}
Thus, using $h$ reduces the understanding of the evolution of $\Omega$ to understanding  a nonlinear system, where the main nonlinearity is in the single term $L^{-1}C_{1,\alpha}\delta^{-\alpha}\int_{-1}^1 \frac{\grad\Omega_*^\perp|_{s_*=s}}{|\xi-\xi_*|^{\alpha}}\dd{\xi}_*  \cdot \grad\Omega$.

For SQG, this term can be expressed in terms of $h$ and $\Omega$, as a bilinear map, making the analysis completely different (see \cite{fefferman2015construct} for more details).%
%

%% file: spinecompatiblecurveevolution.tex
\section{Evolution of compatible curves }
\label{section6}
A sharp front solution to \eqref{eqn-singular-sqg} is completely determined by the evolution of a curve.  In this section, by considering  how compatible curves are transported by \eqref{eqn-singular-sqg} in the regime $\delta\ll 1 $, we show that an almost-sharp front solution is approximately determined by the evolution of a compatible curve.

We will first present the main result of this section, Theorem \ref{prop-comp-curve-evo} by relying on a fractional Leibniz rule \eqref{fractional-leibniz}. 
Lemma \ref{lem-cheap-leibniz}  replaces the more complicated \eqref{fractional-leibniz} at the cost of a small loss in the error term. The proof is similar to  Lemma \ref{lem-fractional-sobolev-norm-of-indicator} (see \cite{faraco2013sobolev} - we present a short proof as the approach will be relevant for Lemma \ref{lem-cheap-leibniz}).

\begin{lem}\label{lem-fractional-sobolev-norm-of-indicator} 	Let $0<s<1/2$ and suppose $A\subseteq\mathbb R^d $ is a bounded set with $C^2$ boundary in $\Real^d$.
 Then $\indicator_A \in H^s(\mathbb R^d)$ with 
	\[\| \hlapl^{s} \indicator_A \|^2_{L^2} \lesssim_{d,s} |A|^{1-2s}.\]
\end{lem}
\begin{proof}  
We bound the Gagliardo seminorm $[\indicator_A]_{H^s}$ directly, which is known (see for instance \cite{dinezza2012hitchhiker}) to be equal to $\|\hlapl^s \indicator_A\|_{L^2}$ up to a constant depending on $s$ alone. By definition,
\[ [\indicator_A]_{H^s}^2 := \int_{\Real^d}\int_{\Real^d}\frac{|\indicator_A(x) - \indicator_A(y)|^2}{|x-y|^{d+2s}}  \dd{y} \dd{x}  = 2\int_{x\in A}\int_{y\in A^c} \frac{\dd{y} \dd{x}}{|x-y|^{d+2s}}.\]
Writing out a `layer cake' decomposition(see for example \cite[page 26]{lieb2001analysis}) with $\mu$ the $2d$-dimensional Lebesgue measure,

\begin{equation*}
    \begin{split}
[\indicator_A]_{H^s}^2 &= \int_0^\infty \mu\left[x\in A,y\in A^c : |x-y| < \frac{1}{t^{1/(d+2s)}}\right]  \dd{t}    \end{split}
\end{equation*}  
From this, we have $[\indicator_A]_{H^s}^2 =\int_0^\infty m\left(\frac{1}{t^{1/(d+2s)}}\right)  \dd{t},$where  $m(t) $ is \[ m(t) := \mu[x\in A,y\in A^c : |x-y| < t ]. \] 
For any set $U\subseteq \mathbb R^d$, define $(U)_\eps = \{x : d(x,U)< \eps  \}$. The change of variables $\tau = \frac{1}{t^{1/(2+2s)}}$, $\dd{t}= (d+2s)\tau^{-d-2s-1}\dd{\tau}$ yields
	\begin{equation*}
	        [\indicator_A]_{H^s}^2 
        = (2d+4s)\int_0^\infty \frac{m(\tau)}{\tau^{d+2s+1}}  \dd{\tau} 
        \lesssim_{d,s} \int_0^{\eps_0}\frac{m(\tau)}{\tau^{d+2s+1}}  \dd{\tau}  + \int_{\eps_0}^\infty \frac{m(\tau)}{\tau^{d+2s+1}}  \dd{\tau}. 
	\end{equation*}
For some $\epsilon_0$ to be chosen as follows. We can bound $m(\tau)$ using the following inclusion:
\[ \left\{ x\in A , y \in A^c : |x-y|<\tau \right\} \subset \left\{ x\in A , y \in B_{\mathbb R^d}(x,\tau)  : d(x,\del A)<\tau   \right\},\]
where $ B_{\mathbb R^d}(x,\tau) $ is the ball around $x$ of radius $\tau$, which implies
\begin{equation}
    \begin{split}
        m(\tau) &
        \le\mu[x\in A, y\in B_{\mathbb R^d}(x,\tau) :  d(x,\del  A)<\tau   ]
				\lesssim |B_{\mathbb R^d}(0,1)| \tau^{d+1} ,
    \end{split}
\end{equation}
since $(\del A)_\tau$ is  $O(\tau)$ by the $C^2$ regularity of the boundary. But for $\tau>|A|$, the following easier bound is better,
\begin{align}
    m(\tau) 
    \le \mu \left[ x\in A : y \in B_{\mathbb R^d} (x,\tau)\right]   
		\lesssim |B_{\mathbb R^d}(0,1)| |A|\tau^d.
\end{align}
 For the optimal bound, we choose $\epsilon_0 = |A|$ to define $I_1,I_2$, giving
\[ \int_0^{\eps_0}\frac{m(\tau)}{\tau^{d+2s+1}}  \dd{\tau} \lesssim_d \int_0^{|A|} \frac{\dd{\tau}}{\tau^{2s}} =\frac1{1-2s}|A|^{1-2s} \lesssim_{s} |A|^{1-2s} \]
and
\[ \int_{\eps_0}^\infty \frac{m(\tau)}{\tau^{d+2s+1}}  \dd{\tau} \lesssim_d |A| \int_{|A|}^\infty \frac{\dd{\tau}}{\tau^{2s+1}} = |A|\frac{|A|^{-2s}}{2s} \lesssim_{s} |A|^{1-2s}. \]
Combining these bounds gives the claimed bound on $[\mathbf 1_A]_{H^s}^2$.
\end{proof} 
We now present the main result of this section.
\begin{thm}\label{prop-comp-curve-evo}
Suppose that $\theta$ is an \ASF solution, and $\mycurve$ is a compatible curve as defined in \eqref{defn-comp-curve}. Then $\mycurve$  satisfies   the sharp front equation (in the weak sense) up to $O(\delta^{1-\alpha})$ errors,
\begin{equation*}
    \del_t \mycurve \cdot \nrml = \left( -\int_{\mathbb T} K(\mycurve-\mycurve_*) (\del_s \mycurve_* - \del_s \mycurve )  \dd{s}_*  \right)
\cdot \nrml+ O(\delta^{1-\alpha}).
\end{equation*}
\end{thm}
\begin{proof} The strategy of this proof is the same as \cite{cordoba2004almost}.
For brevity of notation, we shall in this proof  write 
\[ (\indicator_\init,\indicator_\md,\indicator_\out) 
:= (\indicator_{A_{\init}^{\mycurve}},\indicator_{A_{\md}^{\mycurve}},\indicator_{A_{\out}^{\mycurve}}).\]

As in Proposition \ref{prop-evo-of-SF}, the term $\iint_{\Real^2,t} \del_t \phi \theta $ brings out the time derivative of the $C^2$ boundary curve: since the set $A_{\md}$ has measure $O(\delta)$, 
we see that if $\mycurve^\delta  = \mycurve + \delta \nrml $ parameterises the boundary of $A_{\init}$,

\begin{align}
    \iint_{\Real^2 \times [0,\infty]} \del_t \phi \theta  \dd{x} \dd{t} 
    &= \iint_{\{ x\in\mathbb R^2 , t\ge 0 : x \in A_{\init}(t) \}}  \del_t \phi \dd{x}\dd{t} + O(\delta)
    \\
    &= \int_0^\infty \int_{\mathbb T} \del_t (\mycurve^\delta) \nu^3 \phi(\mycurve^\delta,t) \dd{s}\dd{t} + O(\delta).
\end{align}
Here, $\nu^3$ is the third component of the outward normal $\nu$'s  as a vector in $(x^1,x^2,t)$-space, and
\begin{align}
    \del_t (\mycurve^\delta) \nu^3 = \del_t (\mycurve^\delta) \cdot \del_s ((\mycurve^\delta)^\perp) 
    & 
    = (L-\delta L \kappa ) \del_t (\mycurve + \delta N) \cdot  \del_s\mycurve^\perp    \\
    &
    = L \partial_t \mycurve \cdot \nrml + O(\delta), 
\end{align} 
where we used our definition of a compatible curve. Hence, as $\phi\in C^1$ and $\dd{l} = L\dd{s}$ for the uniform speed parameterised curve $\mycurve$, writing $\del A$ for the curve parameterised by $\mycurve$, we have
\[\iint_{\Real^2 \times [0,\infty]} \del_t \phi \theta  \dd{x} \dd{t} = -\int_0^\infty \int_{\del A}\phi(z) \del_{t}\mycurve\cdot\nrml \dd{l}\dd{t} + O(\delta). \] 
We now treat the second term $\iint_{\Real^2 \times [0,\infty]} \myvel\cdot\grad\phi \theta$ : observe the following decomposition, where we have written $u = \grad^\perp K * \theta = (\grad^\perp K) * \theta$ as a convolution of $\theta$ with the kernel $\nabla^\perp K$, and used the bilinearity of  $(f,g)\mapsto \int_{\Real^2} \grad^\perp K * f\cdot \grad\phi g$, and $\theta = \theta\cdot (\indicator_\init + \indicator_\md + \indicator_\out) = \theta \indicator_\md + \indicator_\init$:

\begin{equation}\label{generic-curve-decomposition}
    \begin{split}
				\int_{\Real^2}  u\cdot \grad \phi \theta 
         &
				= \int_{\Real^2}  \grad^\perp K*\indicator_\init \cdot \grad \phi \indicator_\init  
         + \int_{\Real^2}  \grad^\perp K*  \indicator_\init \cdot \grad \phi  \theta \indicator_\md \\
         &+ \int_{\Real^2}  \grad^\perp K* (\theta \indicator_\md) \cdot \grad \phi \theta \indicator_\md  
         + \int_{\Real^2}  \grad^\perp K* (\theta  \indicator_\md) \cdot \grad \phi \indicator_\init\\
         &=: (\text{EVO}) + (\text{A}) +(\text{B}) +(\text{C}) \\
         &= (\text{EVO}) + [(\text{A}) +(\text{B})] + [(\text{B}) +(\text{C})] - (\text{B}).
    \end{split}
\end{equation}
We will estimate separately each of the 4 terms in the last line of \eqref{generic-curve-decomposition}. Up to $O(\delta^{1-\alpha})$ errors, $(\text{EVO})$ will give us the evolution term, and the square-bracketed terms will use the $C^\eps$ regularity of $\theta = \theta\indicator_\md + \indicator_\init $ that is not available when estimating (A) or (C) alone.

    \subsection*{Control on $[(\text{A}) +(\text{B})]$} We proceed by splitting the kernel $K$,%
    \newcommand{\MYINT}[1]{\int_{\Real^2} {#1}  \cdot \grad\phi \indicator_\md }%
\begin{align*} 
[(\text{A})& +(\text{B})] 
= \MYINT{\grad^\perp K* \theta}  \\
&= \MYINT{(( \grad^\perp K)\indicator_{|\cdot|>\delta}) * \theta}  + \MYINT{((\grad^\perp K)\indicator_{|\cdot|<\delta})* \theta}  
\end{align*} 
We note the bounds (recall that $\int_{|y|<\delta} K(y)\dd{y} = 0$ )
\begin{align*}
    |((\grad^\perp K)\indicator_{|\cdot|>\delta}) * \theta(x) | &\lesssim \|\theta\|_{L^\infty} \int_{r=\delta}^\infty \frac{r\dd{r} }{r^{2+\alpha}} \lesssim_{\theta,\alpha} \delta^{-\alpha} , \text{ and}\\
    |((\grad^\perp K)\indicator_{|\cdot|<\delta}) * \theta(x) | 
    &= \abs{\int_{|y|<\delta} K(y) \theta(x-y) \dd{y} } \\
    &= \abs{\int_{|y|<\delta} K(y)|y|^{\alpha'} \br{\frac{\theta(x-y)-\theta(x)}{|y|^{\alpha'}}} \dd{y} } \\
    &\lesssim [\theta]_{C^{\alpha'}} \int_{r=0}^\delta \frac{r\dd{r}}{r^{2+\alpha-{\alpha'}} } \lesssim [\theta]_{C^{\alpha'}} \delta^{{\alpha'} - \alpha} = O(\delta^{-\alpha}),\label{two-terms}
\end{align*}
where in the second inequality, we used the regularity assumption $[\theta]_{C^{\alpha'}} \lesssim \delta^{-\alpha'}$ for some $\alpha' > \alpha$. Hence both terms  are integrals of $O(\delta^{-\alpha})$ functions that have support of size $O(\delta)$, due to the $\indicator_\md$ term. Therefore, $ (\text{A}) +(\text{B})  = O(\delta^{1-\alpha})$.
\subsection*{ Control on $[(\text{B}) +(\text{C})]$} 
Here, notice that 
\begin{align*} \int_{\Real^2} \grad^\perp K * f \cdot g 
&= -\int_{\Real^2} \del_{x^2} K * f  g^1 + \int_{\Real^2} \del_{x^1} K * f  g^2 \\
&=  -\int_{\Real^2} f \del_{x^2} K *  g^1 + \int_{\Real^2} f \del_{x^1} K *  g^2.
\end{align*}
The important feature is that the two kernels $  \del_{x^1} K(-x),\del_{x^2} K(-x)$ have the same $-2-\alpha$ homogeneity as $\grad^\perp K$, and have mean zero on the unit sphere. Hence, with $f=\theta \indicator_{\md}$ and $g=\grad\phi \theta\in C^\eps$, we can repeat the proof as for $[(\text{A})+(\text{B})]$, obtaining the same $O(\delta^{1-\alpha})$ estimate.
\subsection*{ Control on $(\text{B})$} Writing $R = \nabla \Lambda^{-1}$ for the vector of Riesz transforms (see for instance \cite{elias1970singular}), we have $\nabla^\perp K *f= R^\perp \Lambda^{\alpha}f $, so that
    \begin{equation*}
        \begin{split}
             |(\text{B})| 
             &= \left|\int_{\Real^2}R^\perp \hlapl^{\alpha/2}  (\theta \indicator_{\md}) \cdot \hlapl^{\alpha/2}(\grad \phi  \theta \indicator_\md) \right| \\
             &\lesssim \norm{ R^\perp \hlapl^{\alpha/2} (\theta \indicator_\md)  }_{L^2} 
             \norm{ \hlapl^{\alpha/2}( \grad \phi  \theta \indicator_\md) }_{L^2} \\
             &\lesssim \norm{ \hlapl^{\alpha/2} (\theta \indicator_\md)  }_{L^2} 
             \norm{ \hlapl^{\alpha/2}( \grad \phi  \theta \indicator_\md) }_{L^2},
        \end{split}
    \end{equation*}
    where the last line uses the boundedness of $R^\perp : L^2\to L^2$. To bound these terms, we will use the following fractional Leibniz rule  of e.g. \cite{kenig1993well-posedness},
    \begin{align}
        \|\Lambda^{s} (fg)\|_{L^2} \le \|\Lambda^{s} f\|_{L^\infty} \|g\|_{L^2} + \|f\|_{L^\infty}\|\Lambda^{s} g\|_{L^2} \label{fractional-leibniz} \quad s\in (0,1),
    \end{align} 
and the following easy estimate, valid for $s+\epsilon \le 1$ that comes from bounding the following two terms separately (similarly to the earlier part of this proof) $\Lambda^{s} f(x)= \left(\int_{|x-y|\le \delta} +\int_{|x-y|>\delta} \right)\frac{f(x)-f(y)}{|x-y|^{2+s}} \dd{y}$,
\[ \|\Lambda^{s} f(x)\|_{L^\infty} \lesssim  \delta^{\epsilon} [f]_{C^{s+\epsilon}}  + \delta^{-s} \|f\|_{L^\infty}. \]
By interpolation, since $|\nabla \theta|\lesssim \frac1{\delta}$ (Definition \ref{defn-ASF}), $[f]_{C^{s+\epsilon}}\lesssim \delta^{-s-\epsilon}$ for $s+\epsilon\le 1$. Setting $s = \alpha/2$, $g = \mathbf 1_{\md}$  and $f = \theta$ or $\nabla\phi \theta$, we obtain
    \[ \|\hlapl^{\alpha/2} (\theta \indicator_{\md} ) \|_{L^2} \lesssim \delta^{(1-\alpha)/2}, \qquad \qquad
			            \norm{ \hlapl^{\alpha/2}( \grad \phi  \theta \indicator_\md) }_{L^2}
              \lesssim 
             \delta^{(1-\alpha)/2}. 
							\]		
    Together, these inequalities prove that $|(\text B)|\lesssim \delta^{1-\alpha}$. 
\subsection*{Evolution term in $(\text{EVO})$}
  By following the proof of the analogous sharp front result \eqref{prop-evo-of-SF} and using $\mycurve^\delta  = \mycurve + \delta \nrml $ again,\begin{align*}
I_{\init} &=  \int_{\mathbb T} \phi(\mycurve^\delta) \left( \int_{\mathbb T} K(\mycurve^\delta-\mycurve^\delta_*) \del_s \mycurve^\delta_*\cdot\del_s (\mycurve^\delta_*)^\perp   \dd{s}_*  \right)
  \dd{s} \\
     &= \int_{s\in I} \phi(\mycurve) \left( \int_{\mathbb T} K(\mycurve-\mycurve_*) (\del_s \mycurve_* - \del_s \mycurve )  \dd{s}_*  \right)
    \cdot \nrml \dd{s}+ O(\delta).
\end{align*}
This last line follows from a simple application of Mean Value Theorem, completing the proof.
\end{proof}

In the above proof, we relied on a fractional Leibniz rule \eqref{fractional-leibniz}. The following lemma can serve as a weak replacement:

\begin{lem}[H\"older-Indicator Leibniz Rule]\label{lem-cheap-leibniz}
Let $0<s<1/2$, $s<s'$. Suppose $A\subseteq \mathbb R^d$ is a bounded set 
with $C^2$ boundary and let $f \in C^{s'}(A)$.  Then the extension of $f$ by zero, $ f \indicator_A $ belongs to  $H^{s}(\Real^d)$ with 
		\[\| \hlapl^{s} (f\indicator_A) \|^2_{L^2(\Real^d)} \lesssim_{d,s} \|f\|^2_{L^\infty(A)} |A|^{1-2s	} + [f]^2_{C^{s'}(A)}|A|^{1-\frac{2(s-s')}{d}}. \]
 
\end{lem}
\begin{proof}
We again bound the Gagliardo seminorm $[f\indicator_A]_{H^s}^2$. We have
\begin{align}
    [f\indicator_A]_{H^s}^2 &= \left( 2\int_{A^c}\int_{ A}+\int_{A}\int_{ A}  \right) \frac{|(f\indicator_A)(x) - (f\indicator_A)(y)|^2}{|x-y|^{d+2s}}  \dd{x} \dd{y} 
    \\
    &=  2\int_{A^c}\int_{ A} \frac{|f(x)|^2}{|x-y|^{d+2s}}  \dd{x} \dd{y} 
     + \int_{A}\int_{ A} \frac{|f(x) - f(y)|^2}{|x-y|^{d+2s}}  \dd{x} \dd{y}. 
\end{align}
the proof of Lemma \ref{lem-fractional-sobolev-norm-of-indicator} (when only one of $x,y$ is in $A$) yields
\[  \int_{A^c}\int_{ A} \frac{|f(x)|^2}{|x-y|^{d+2s}}  \dd{x} \dd{y} \lesssim_{d,s} \|f\|^2_{L^\infty} |A|^{1-2s} ,\]
but unlike Lemma \ref{lem-fractional-sobolev-norm-of-indicator}, we do have contributions when $(x,y) \in A\times A$. That is, we need to obtain a bound on the  integral
\begin{equation}
    \begin{split}
         I_{A} &:= \int_{A}\int_{ A} \frac{|f(x) - f(y)|^2}{|x-y|^{d+2s}}  \dd{x} \dd{y}.
    \end{split}
\end{equation} 
Using a layer cake decomposition again with \[\tilde m(t) := \mu(x,y\in A, y\in A : |x-y|<t),\]
we can estimate $I_A$ as follows:
\begin{equation}
    \begin{split}
        I_A 
        &= \int_0^\infty \mu\left[x\in A,y\in A : \frac{|f(x) - f(y)|^2}{|x-y|^{d+2s}} > t\right]  \dd{t}\\
        &\le \int_0^\infty \mu\left[x\in A,y\in A :\frac{[f]^2_{C^{s'}}}{|x-y|^{d+2(s-s'))}} > t \right]  \dd{t}   \\
        &= \int_0^\infty \tilde m\! \br{([f]_{C^{s'}}^2/t)^{\frac1{d+2(s-{s'})}} }  \dd{t}
				\lesssim_{d,s,s'}  [f]_{C^{s'}}^2 \int_0^\infty \frac{\tilde m(\tau) }{\tau^{d+2(s-{s'})+1} }  \dd{\tau} ,
    \end{split}
\end{equation}
where in the last line we have changed variables $\tau = ([f]^2_{C^{s'}} / t)^{\frac1{d+2(s-{s'})}}$%
 and ignored constants. Observe  that by reasoning similarly to  Lemma \ref{lem-fractional-sobolev-norm-of-indicator}, 
 \[\tilde m(\tau) \lesssim_d \min( |A|^2, |A| \tau^d).\] 
 Hence, the optimal bound is obtained by splitting the integration region $\tau > 0$ into the sets $\tau \in[0,|A|^{1/d}]$ and  $\tau \in [|A|^{1/d},\infty]$, which yields 
\begin{equation}
	\begin{split}
		I_A 
		&\lesssim  [f]^2_{C^{s'}}  |A| \int_0^{|A|^{1/d}} \frac{\dd{\tau}}{\tau^{2(s-{s'})+1}} 
		+ [f]^2_{C^{s'}} |A|^2 \int_{|A|^{1/d}}^\infty \frac{\dd{\tau}}{\tau^{d+2(s-{s'})+1}}  \\
		&= [f]^2_{C^{s'}} ( |A|^{1-\frac{2(s-s')}{d}} + |A|^{2-\frac{d+2(s-s')}{d}})\lesssim [f]^2_{C^{s'}}  |A|^{1-\frac{2(s-s')}{d}},
	\end{split}
\end{equation} 
so long as $s'> s$. As there is no contribution to $[f\indicator_A]_{H^s}^2$ when $(x,y)\in A^c \times A^c$, this concludes the proof.
\end{proof}%
Using Lemma \ref{lem-cheap-leibniz} instead of the fractional Leibniz rule \eqref{fractional-leibniz}, we obtain the following weaker result.
\begin{thm}\label{prop-comp-curve-evo-weaker}
Suppose that $\theta$ is an \ASF solution, and $\mycurve$ is a compatible curve as defined in \eqref{defn-comp-curve}. Then $\mycurve$  satisfies   the sharp front equation (in the weak sense) up to $O(\delta^{1-\alpha'})$ errors for any $\alpha'>\alpha$,
\begin{equation*}
    \del_t \mycurve \cdot \nrml = \left( -\int_{\mathbb T} K(\mycurve-\mycurve_*) (\del_s \mycurve_* - \del_s \mycurve )  \dd{s}_*  \right)
\cdot \nrml+ O(\delta^{1-\alpha'}).
\end{equation*}
\end{thm}

%% file: spineevolution.tex
\section{The spine of an almost-sharp front}\label{section-spine}
Here we introduce the concept of the spine,  first considered by Fefferman, Luli and Rodrigo in \cite{fefferman2012spine} for SQG in the periodic setting (our contours are however not graphs). 
To simplify the following calculations, assume without loss of generality  $\Omega$ is given by $\eqref{eqn-omega-pmonehalf}$.

\begin{defn}  \label{defn-spine}Suppose an almost-sharp  front has tubular neighbourhood coordinates  (see Section \ref{tub-coords}) $(s,\xi)$ for the transition region,  induced by the compatible curve $\mycurve$. 

We say that the curve $\myspine$ is a spine for the almost-sharp front if $\myspine$ is also a compatible curve, and there is a $C^2 $ function of the uniform speed parameter $f=f(s)$ taking values in $[- C^{\mycurve},C^{\mycurve}]$ such that 
\begin{equation}
	\int_{-C^{\mycurve}}^{C^{\mycurve}} (\xi_* - f(s_*))\del_\xi\Omega_* \dd{\xi}_* = 0,
\end{equation} 
or equivalently by the choice $\Omega|_{\xi=\pm C^{\mycurve} } = \pm 1/2    $, $f(s_*)= - \int_{-C^{\mycurve}}^{C^{\mycurve}}  \Omega_* \dd{\xi}_*$, and the corresponding spine is the curve $\myspine$ given in $(s,\xi)$ coordinates as $\xi = f(s)$, that is:
\[ \myspine(s) = \mycurve(s) + \delta f(s)\nrml(s). \]
\end{defn}
The function $f$ acts as a correction, so that for example, the base curve is also a spine if $f=0$.

An immediate consequence of Definition \ref{defn-spine} by integrating by parts is the following cancellation property     for any constant $C\ge  \|f\|_{L^\infty}$,
\begin{align}\label{eqn-spine-cancellation} 
 0 = \int_{\xi_*=f-C - C^{\mycurve} }^{f+C+C^{\mycurve}} \Omega_* \dd{\xi}_*, 
 \end{align}  
 where $\Omega$ is continuously extended to be constant on $|\xi|\ge  C^\mycurve    $ past the geometrically significant range that defines the tubular neighbourhood. Indeed, for $D\gg 1$, 
 \[ \int_{f-D}^{f+D} \!\!\Omega_* 
\dd{\xi_*} 
= \underbrace{
    \xi_* \Omega_* \Big|_{f-D}^{f+D}
    }_{
        =-f}   
    - \underbrace{
    \int_{f-D}^{f+D} \!\! \xi_* \del_\xi \Omega_* \dd{\xi_*}
    }_{
        = \int_{-C^{\mycurve}}^{C^{\mycurve}} \xi_* \partial_\xi \Omega_* \dd{\xi_*}
        } 
= \int_{-C^{\mycurve}}^{C^{\mycurve}} (\xi_* - f)\del_\xi\Omega_* \dd{\xi_*} = 0. \]
We also have the following property. 

\begin{prop}[Spine approximation property]\label{spine-property}
Let $\theta$ be a $C^2$ almost-sharp front, and let $\myspine$ be the spine curve defined by the above \eqref{defn-spine}. Then for any $\Gamma = \Gamma(x) \in C^2_c(\Real^2)$, as $\delta\to0$,
\begin{align}\label{eqn-spine-property}\int_{x\in\Real^2} \Gamma \grad^\perp\theta \dd{x} = -\int_{\mathbb T} \Gamma(\myspine_*) \del_s \myspine_*  \dd{s}_* + O(\|\Gamma\|_{C^2}\delta^{2}).\end{align}
\end{prop}
We keep the $\|\Gamma\|_{C^2}$ dependence  in \eqref{eqn-spine-property} because we will require the use of test functions with $\|\Gamma\|_{C^2}$ that degenerate as $\delta\to 0$.

\begin{proof}
Write $\tilde\Gamma = \Gamma(x(s,\xi))$, with $x=\mycurve + \delta \xi \nrml$. We are looking for the following expansion 
\begin{equation}
\label{wanted-spine-expansion}
     \iint  \br{\delta \del_s \Omega_* \nrml_* - L(1-\delta\kappa_*\xi_*)	\del_\xi\Omega_* \tgt_*} \tilde\Gamma_* \dd{s}_* \dd{\xi}_*
     + \int_{\mathbb T} \tilde\Gamma_*\del_s\myspine_*	 \dd{s}_*
     \overset{?}{=} O(\delta^2).
\end{equation}

We Taylor expand around $\myspine_*$ in the direction of $\nrml_*$ to obtain
\begin{equation*}
	\begin{split}
		\tilde\Gamma_* = \Gamma(x_*) &= \Gamma(\myspine_*) + \grad\Gamma(\myspine_*)\cdot (x_* - \myspine_*) + O(|x_*-\myspine_*|^2) \\
		&= \Gamma(\myspine_*) + \delta \grad\Gamma(\myspine_*) \cdot  \nrml_* (\xi_* - f_*) + O(\delta^2).
	\end{split}
\end{equation*} 
Plugging into the left hand side (LHS) of \eqref{wanted-spine-expansion} we have
 \begin{equation}\begin{split}\label{LHS-spine}
&\text{LHS}\\ 
&= \int_{\mathbb T} \Bigg[  \Gamma(\myspine_*) \br{\del_s\myspine_*	 -  \tgt_* \int_{-1}^1L(1-\delta\kappa_*\xi_*) \del_\xi\Omega_* \dd{\xi}_* + \delta \nrml_*\int_{-1}^1  \del_s \Omega_*  \dd{\xi}_*} \\
&  - L \delta \grad\Gamma(\myspine_*) \cdot  \nrml_* \int_{-1}^1 (\xi_* - f_*)   \del_\xi\Omega_* \tgt_*  \  \dd{\xi}_* \Bigg] \dd{s}_* + O(\delta^2) .
\end{split}
\end{equation}
The following identities (immediate from the definitions of $\myspine$ and $f$), 
\begin{align*}
0 &= f'_* + \int_{-1}^1  \del_s \Omega_*  \dd{\xi}, \\
0&= \int_{-1}^1 (\xi_* - f_*)\del_\xi\Omega_* \dd{\xi}, \text{ and} \\
\del_s\myspine_* &=  L(1 - \delta\kappa_* f_*) \tgt_*  + \delta f'_*\nrml_* \\
&=  L(1 - \delta\kappa_* f_*) \tgt_* \int_{-1}^1 \del_\xi \Omega_* \dd{\xi} + \delta f_*'\nrml_*,
\end{align*}
show that the right-hand side of \eqref{LHS-spine} is of order $\delta^2$, as the terms in the square brackets vanish.
\end{proof}

\subsection{Evolution of a spine}
Theorem \ref{prop-comp-curve-evo} showed that any compatible curve evolves by the sharp front equation \eqref{eqn-generic-cde} up to an error of order $O(\delta^{1-\alpha})$. However, for the spine, we will  improve this to  $O(\delta^{2-\alpha})$.

\begin{thm}
Suppose that $\theta$ is an ASF solution. Then the spine curve $\myspine$  in Definition \ref{defn-spine} evolves (in the weak sense) according to the sharp front equation up to $O(\delta^{2-\alpha})$ error. That is,
\begin{equation}
    \del_t \myspine \cdot \nrml  = \left(- \int_{s_*\in I} K(\myspine-\myspine_*) (\del_s \myspine_* - \del_s \myspine )  \dd{s}_*  \right)
\cdot \nrml + O(\delta^{2-\alpha}) \label{eqn-spine-evolution}.
\end{equation}
\end{thm}
\begin{proof}
Without loss of generality, assume the constant for the base curve $C^{\mycurve} = 1$, and
  choose $\delta \ll 1$ so that we can extend the $\xi$ coordinate to the range $[-3,3]$, having a well defined neighbourhood of thickness $6\delta$.  Recall that $\theta$ is a weak solution if for every $\Gamma \in C^\infty_c((0,\infty) \times \Real^2)$,
\begin{align} 0 = \iint_{t\ge 0,x\in \Real^2} \theta \del_t \Gamma  + \theta (u\cdot \grad \Gamma)  \dd{x} \dd{t}.
\label{weaksoln+gammathing}
\end{align}

Define for each time $t$ the spine curve $\myspine=\myspine(s,t) \in \myspinein(t)$, the inner region bounded by the closed curve $\myspine + 2\delta \nrml$ with $\nrml$ the inward normal, the outer region $\myspineout(t)$ bounded by $\myspine - 2\delta \nrml$, and the tubular region $\myspinemid(t)$ in the middle of radius $2\delta$.   We give the names $\myspine^+$ and $\myspine^-$ to the inner and outer boundary curves of $\myspinemid$ respectively,
 \begin{align}
 \myspine^{+} &:= \myspine +2\delta \nrml = \mycurve + \delta(f + 2 )\nrml, \\
 \myspine^{-} &:= \myspine - 2\delta \nrml = \mycurve + \delta(f - 2 )\nrml    .
 \end{align}
 We thus have for each $t$ (up to null sets),
\begin{equation*}
    \begin{split}
        \Real^2 &= \myspinein(t) \cup \myspineout(t) \cup \myspinemid(t),\\ \myspinein(t) &\subseteq \{x : \theta(x,t) = +1/2\},\\ 
        \myspineout(t) &\subseteq \{x : \theta(x,t) = -1/2 \}.
    \end{split}
\end{equation*} 
Also define the related partition of $\Real^2 \times [0,\infty)$ by 
\begin{align*}
        \init
				= \bigcup_{t\geq 0} \myspinein(t) \times \set{t},
			\  
		\out
			= \bigcup_{t\geq 0}  \myspineout(t) \times \set{t}, \
		\md
				= \bigcup_{t\geq 0}  \myspinemid(t) \times \set{t}.
\end{align*} 
We treat the two integrands $\theta \del_t \Gamma $ and  $\theta (u\cdot \grad \Gamma)$ in \eqref{weaksoln+gammathing} separately, with the three sets to integrate over. For the first integrand, we have 
 \[ \br{\iint_{\init} + \iint_{\out} + \iint_{\md} } \theta \del_t \Gamma  \dd{x} \dd{t}=: I_\init + I_\out + I_\md. \]
In the tubular region, with $ L_1(s,\xi) = L(1-\delta \kappa(s) \xi) = L + O(\delta)$,
\begin{align*}
I_{\md} &= \iint_{t\ge 0,s\in\mathbb T }\int_{\xi = f -2}^{f + 2} \Omega(s,\xi,t) \del_t \Gamma(x(s,\xi),t) \delta L_1(s,\xi)  \dd{\xi}  \dd{s}\dd{t}\\
     &= \iint_{t\ge 0,s\in\mathbb T }\int_{\xi = f -2}^{f + 2} \Omega  \del_t\Gamma(x(s,\xi),t)L  \delta  \dd{\xi} \dd{s}\dd{t} + O(\delta^2)
.
\end{align*} By the spine cancellation property \eqref{eqn-spine-cancellation} we have
  \begin{align*}
  I_\md 
    &= \delta \iint_{t\ge 0,s\in\mathbb T }\int_{\xi = f -2}^{f + 2} \Omega [ \del_t\Gamma(x,t) -   \del_t\Gamma(\myspine,t)] L \dd{\xi} \dd{s}\dd{t} + O(\delta^2)
    \\&
    = O(\delta^2), 
  \end{align*}
since $|\del_t\Gamma(x(s,\xi),t)-\del_t \Gamma(\myspine(s),t)|\lesssim |x(s,\xi)-\myspine(s)| = O(\delta) $ uniformly in $s$ and $\xi$. 
For $I_\init$, we apply the Divergence theorem in 3D,
\begin{align*}
\iint_{(x,t)\in \init} \theta \del_t\Gamma  \dd{t}\dd{x} &= \frac{1}{2} \iint_{(x,t)\in \init} \grad_{x^1,x^2,t} \cdot {\tiny\begin{bmatrix}
    0 \\ 0 \\ \Gamma 
\end{bmatrix}}  \dd{t}\dd{x}  \\
&= \frac{1}{2}\iint_{t\ge 0,s\in\mathbb T } \Gamma(\myspine^+ , t) \del_t[\myspine^+] \cdot \del_s[\myspine^+]^\perp \dd{s}\dd{t}. 
\end{align*}
Similarly for $I_\out$ we obtain the term (note the minus sign from the opposite orientation)
\begin{align*}
 \iint_{(x,t)\in \out} \theta \del_t\Gamma  \dd{t}\dd{x}  &= \frac{-1}{2}\iint_{t\ge 0,s\in\mathbb T } \Gamma(\myspine^-, t) (-\del_t[\myspine^-]) \cdot \del_s[\myspine^-]^\perp \dd{s}\dd{t}\\
 &= \frac{1}{2}\iint_{t\ge 0,s\in\mathbb T } \Gamma(\myspine^-, t) \del_t[\myspine^-] \cdot \del_s[\myspine^- ]^\perp \dd{s}\dd{t}.\end{align*}
Since $\myspine^\pm = \myspine \pm O(\delta)$, we obtain by the approximation formula valid for $C^2$ functions,
$f(a+b) + f(a-b) = 2f(a) + O(b^2),    $
(with $O(b^2)$ constant depending on $\|f''\|_{L^\infty}$) that
\[ \iint_{\init \cup \out} \theta \del_t\Gamma  \dd{t}\dd{x} = \iint_{t\ge 0,s\in\mathbb T } \Gamma(\myspine) \del_t \myspine \cdot \del_s \myspine^\perp \dd{s}\dd{t}+ O(\delta^2), \]
with $O(\cdot)$ constant depending on $\| \Gamma \|_{C^2}$ and the geometry of the base curve $\mycurve$. We therefore obtain that $I_\init + I_\out + I_\md 
= \iint_{t\ge 0,x\in\mathbb R^2}  \theta \del_t\Gamma  \dd{t}\dd{x}$ can be written as
\begin{align*} I_\init + I_\out + I_\md 
&= \iint_{t\ge 0,s\in\mathbb T } \Gamma(\myspine) \del_t \myspine \cdot \del_s \myspine^\perp \dd{s}\dd{t}+ O(\delta^2). \end{align*}
For the second term, define $B(t)$ as the following integrand,
\begin{equation}
\int_{x\in\Real^2,t\ge 0} \theta u \cdot\grad\Gamma  \dd{x} \dd{t}=: \int_{t\ge 0} B(t)  \dd{t}.
\end{equation}
We need to control $B(t)$.  We symmetrise the integrand of $B(t)$ by reversing the roles of $x$ and $y$, obtaining
\[ 2B(t) = \int_{x\in \Real^2}\int_{y \in \Real^2} \grad_x^\perp |x-y|^{-1-\alpha} \cdot[\grad\Gamma(x) - \grad\Gamma(y)]\theta(x)\theta(y)  \dd{x}\dd{y}.    
\]
 We split $\Real^2\times\Real^2 = C_{0} \cup C_{1} \cup C_{2}$, where the subscript in $C_i$ depends on whether none of, one of, or both of $x$  and $y$ are in $\md$ respectively,
\begin{equation}
\begin{split}
    C_0 &=  ({\md})^c\times ({\md})^c, \\
    C_1 &=  \left[{\md}\times ({\md})^c\right] \cup \left[ ({\md})^c \times {\md}\right],  \\
    C_2 &=  {\md}\times {\md}.
\end{split}
\end{equation}  
Then write $B_0(t) +B_1(t) + B_2(t) = B(t)$, where 
 \[ B_i(t) = \frac12 \iint_{(x,y)\in C_i} \grad_x^\perp |x-y|^{-1-\alpha} \cdot [\grad\Gamma(x) - \grad\Gamma(y)]\theta(x)\theta(y)   \dd{x}\dd{y}, \] 

 $B_1$ is an error term. Indeed, \begin{align*}
 2B_1(t)
 &=  \iint_{(x,y)\in C_1} \grad_x^\perp |x-y|^{-1-\alpha}\cdot [\grad\Gamma(x) - \grad\Gamma(y)]\theta(x)\theta(y)   \dd{x}\dd{y} \\
 &= \br{ \iint_{x\in \md , y \in (\md)^c} + \iint_{x\in (\md)^c , y \in \md}}  \grad_x^\perp |x-y|^{-1-\alpha}\cdot\\
 & \quad [\grad\Gamma(x) - \grad\Gamma(y)]\theta(x)\theta(y)   \dd{x}\dd{y}\\
 &=: B_{11} + B_{12}.
\end{align*}
We show below that $B_{12}$ is of order $O(\delta^{2-\alpha})$; $B_{11}$ can be estimated similarly. Integrating by parts the derivative in $\nabla^\perp_x |x-y|^{-1-\alpha}$ and using the product rule $\grad^\perp \cdot (f V) = \grad^\perp f \cdot V + f\grad^\perp \cdot V$, we are left with only the boundary terms because $\grad^\perp \cdot \grad \Gamma = 0$ and also $\grad^\perp \theta|_{x\in(\md)^c} = 0$. To get the sign on the integration by parts correct, note 
\begin{align*} \int_{ \md^c} \nabla\cdot  F(x) \dd{x} 
&= \int_{\del \md^c}  F(x)\cdot N_{\out,\del \md^c}(x) \dd{x}
\\
&= \int_{\mathbb T} F(\myspine^-_* )\cdot (\del_s \myspine^-_*)^\perp \dd{s_*} - \int_{\mathbb T} F(\myspine^+_* )\cdot (\del_s \myspine^+_*)^\perp \dd{s_*}\\
&= -\sum_{\sigma = \pm 1} \int_{\mathbb T} F(\myspine^\sigma_* )\cdot (\del_s \myspine^\sigma_*)^\perp \sigma \dd{s_*} 
  \end{align*}
  So as $\theta(\myspine^\sigma) = \frac{\sigma}2$, $\nabla^\perp \cdot F = \nabla \cdot (-F^\perp)$, and $\nabla\theta |_{\del\md} = 0$, we have the following versions of Divergence Theorem:
\begin{align}
  \int_{ \md^c} \nabla^\perp \cdot  F(x) \theta(x) \dd{x} 
  &= \frac12\sum_{\sigma = \pm 1} \int_{\mathbb T} F(\myspine^\sigma_* )\cdot (\del_s \myspine^\sigma_*)  \dd{s_*}   , \label{eqn-signsgradperp} \\ 
     \int_{ \md^c} \nabla \cdot  F(x) \theta(x) \dd{x} 
  &= -\frac12\sum_{\sigma = \pm 1} \int_{\mathbb T} F(\myspine^\sigma_* )\cdot (\del_s \myspine^\sigma_*)^\perp  \dd{s_*}  \label{eqn-signsgrad}
 .
\end{align}
Therefore, we can write $B_{12}$ as follows:
\begin{align*}
&B_{12} \\
&= \int_{\md} \theta(y) \int_{\del\md^c} |x-y|^{-1-\alpha}  [\grad\Gamma(x) - \grad\Gamma(y)] \cdot  \tgt_{\del \md^c}(x) \theta(x)   \dd{l(x)}\dd{y}  \\
&= \int_{\md} \frac{\theta(y)}{2} \underbrace{\sum_{\sigma = \pm 1} \int_{\mathbb T} |\myspine^\sigma_* -y|^{-1-\alpha}  [\grad\Gamma(\myspine^\sigma_* ) - \grad\Gamma(y)] \cdot\del_s\myspine^\sigma_*\  \dd{s}_*}_{=:G(y)}  \dd{y}  \end{align*} 
Writing the $y$-integral in tubular coordinates around $\myspine$, we can use the cancellation identity \eqref{eqn-spine-cancellation} to see that
\[ 2B_{12} = \int_{\mathbb T}\int_{\xi=f -2}^{f +2} \Omega(s,\xi) (G(y(s,\xi)) - G(\myspine)) \delta L  \dd{s}\dd{\xi} + O(\delta^2).\]
In what follows, we will write $y$ for the parameterised point $y=y(s,\xi)$. 
If we can prove $|G(y) - G(\myspine)|\lesssim \delta^{1-\alpha}$,  this would imply that $B_{12} = O(\delta^{2-\alpha})$. We now use a smooth cut-off function $\rho_\delta(s) = \rho(s/\delta)$ with $ \supp \rho = [-1,1]$, $\rho|_{[-1/2,1/2]} = 1$ to split $G(y)$ into two parts, %
%
\begin{align*}
    G(y) &= \sum_{\sigma = \pm 1} \int_{\mathbb T} \underbrace{|\myspine^\sigma_* - y|^{-1-\alpha}[\grad \Gamma (\myspine^\sigma_*) - \grad\Gamma(y)]}_{\mathcal G(\myspine_*^\sigma,y)}    \cdot \del_{s_*}(\myspine^\sigma_*)  \dd{s}_* \\
    &= \sum_{\sigma = \pm 1} \int_{\mathbb T} \rho_\delta (s-s_*)  \auxG(\myspine_*^\sigma , y) \cdot \del_{s_*}(\myspine^\sigma_*)  \dd{s}_*  \\
    &\quad + \sum_{\sigma = \pm 1} \int_{\mathbb T} \left( 1 -  \rho_\delta (s-s_*)  \right) \auxG(\myspine_*^\sigma , y)  \cdot \del_{s_*}(\myspine^\sigma_*)  \dd{s}_*  \\
    &=: G_1(y) + G_2(y).
\end{align*}
Note the function $\auxG = \auxG(x,y)$ defined by the above lines, with $x=\myspine_*^\sigma$.
For $G_1$, the support  in $s_*$ of $\rho_\delta (s-s_*) $ gives us the required control using $|\auxG(x,y)| \lesssim_\Gamma |x-y|^{-\alpha}$, so that $|\auxG(\myspine_*^\sigma,y)| = O(|s_*-s|^{-\alpha})$. Therefore, we have the bound
\[ |G_1| \le \|\rho\|_{L^\infty} \|\del_s \myspine\|_{L^\infty} \| \auxG(\myspine_*^\sigma , y)\|_{L_{s_*}^1[s-\delta,s+\delta]} = O_{\rho,\Gamma} (\delta^{1-\alpha}).\]
So it now suffices to study the derivative of $G_2$, since
\begin{align*}
|G(y)-G(\myspine)| &\lesssim |G_2(y)-G_2(\myspine)| + O(\delta^{1-\alpha}) \\ & \hspace{-2em}  \lesssim \|\nabla G_2\|_{L^\infty} |y-\myspine| + O (\delta^{1-\alpha}) = \|\nabla G_2\|_{L^\infty} O(\delta) + O (\delta^{1-\alpha})  .
\end{align*}
In the tubular coordinates $y=y(s,\xi)$, we have to control the two terms $\del_s G_2(y(s,\xi))$ and $\del_\xi G_2(y(s,\xi))/\delta $. The first term is
\begin{align*}
 \del_s G_2(y) 
&= \sum_{\sigma = \pm 1} \int_{\mathbb T} -\del_{s_*} \br{1- \rho_\delta (s-s_*) } \auxG(\myspine_*^\sigma , y)  \cdot \del_{s_*}(\myspine^\sigma_*)  \dd{s}_* \\
&+ \sum_{\sigma = \pm 1} \int_{\mathbb T} \br{1- \rho_\delta (s-s_*) } \del_s\auxG(\myspine_*^\sigma , y)  \cdot \del_{s_*}(\myspine^\sigma_*)  \dd{s}_*     
\\
&= \sum_{\sigma = \pm 1} \int_{\mathbb T} \br{1- \rho_\delta (s-s_*) }( \del_{s_*}\auxG(\myspine_*^\sigma , y)+ \del_{s}\auxG(\myspine_*^\sigma , y))  \cdot \del_{s_*}(\myspine^\sigma_*)  \dd{s}_* \\
&+ \sum_{\sigma = \pm 1} \int_{\mathbb T} \br{1- \rho_\delta (s-s_*) } \auxG(\myspine_*^\sigma , y)  \cdot \del^2_{s_*}(\myspine^\sigma_*)  \dd{s}_* .
\end{align*}
The worst term $\del_{s_*}\auxG(\myspine_*^\sigma , y)+ \del_{s}\auxG(\myspine_*^\sigma , y)$ is $ O(|s_*-s|^{-1-\alpha})$, and the cutoff function restricts the integration to the region $|s-s_*|\ge\delta/2$. Thus, $\partial_s G_2 = O(\delta^{-\alpha})$.

The other derivative $\frac{\del_\xi }{\delta} G_2(y)$ is simpler since the cutoff $\rho_\delta$ does not depend on $\xi$, 
\begin{align*}
\frac{\del_\xi}{\delta} G_2(y) = \sum_{\sigma = \pm 1}  \int_{\mathbb T} \left( 1 -  \rho_\delta (s-s_*)   \right)  \frac{\del_\xi}{\delta}\underbrace{\br{ \frac{\grad \Gamma (\myspine^\sigma_*) - \grad\Gamma(y)}{|\myspine^\sigma_* - y|^{1+\alpha}}}}_{=:G_{2,1}(y)}\cdot \del_{s_*}\myspine^\sigma_*  \dd{s}_* 
\end{align*}
Note that $\frac{\del_\xi}{\delta}G_{2,1}(y) = \nabla_y G_{2,1}(y) \cdot \frac{\del_\xi}{\delta} y$, and  $\frac{\del_\xi}{\delta} y = \nrml = O(1)$; so we have an  $O(|s-s_*|^{-1-\alpha})$ integrand, integrated on the region $|s-s_*|>\delta/2$. Hence, $G_2(y) = O(\delta^{2-\alpha})$,  so $B_{11},B_{12}= O(\delta^{2-\alpha})$, and therefore
\[ B_1 = O(\delta^{2-\alpha}).   \]
\begin{align} 
 &2B_2(t)  \\
 &=  \iint_{(x,y)\in C_1} \grad_x^\perp |x-y|^{-1-\alpha} \cdot [\grad\Gamma(x) - \grad\Gamma(y)]\theta(x)\theta(y)   \dd{x}\dd{y} \\
 \label{similar-computation} &= -\frac12\sum_{\sigma=\pm 1}\int_{\mathbb T}\int_{\md} \theta(y) |\myspine^\sigma_*-y|^{-1-\alpha}   [\grad\Gamma(\myspine^\sigma_*) - \grad\Gamma(y)]\cdot \del_{s_*} (\myspine^\sigma_*)     \dd{s}_* \dd{y} \\
 & \quad - \int_{\md } 
    \underbrace{
        \br{
        \int_{\md }  |x-y|^{-1-\alpha}  [\grad\Gamma(x) - \grad\Gamma(y)]\theta(y)   \dd{y}
        }%
    }_{=:Q(x)}%
    \cdot \grad^\perp\theta(x) \dd{x} \label{eqn-here-we-use-error-estimate}
 \end{align}
Above, the $\grad^\perp_x$ never falls on $\grad\Gamma$ due to $\grad\cdot\grad^\perp = 0$. For \eqref{similar-computation}, the singularity is no worse than the one for $B_1$ and can be treated in exactly the same way. For \eqref{eqn-here-we-use-error-estimate}, we aim to use  \eqref{eqn-spine-property} of Proposition \ref{spine-property}, so we need to estimate  $\| \nabla^2 Q\|_{C^2}$. In what follows, we concatenate vectors to denote a tensor e.g. $(UWV)_{ijk} = U_i W_j V_k $, and $(\nabla^2 F)_{ijk} = \del_i\del_j F_k$. Integrating by parts twice, and introducing the term $r=r(x)$,
\begin{align}
 \nabla^2 Q (x) 
 &= \int_{\md } \nabla^2_x \left( |x-y|^{-1-\alpha}  [\grad\Gamma(x) - \grad\Gamma(y)]\right) \theta(y)  \dd{y} \label{eqn-522}
 \\
 &\hspace{-3em}=      \int_{\md } \nabla^2_y \left( |x-y|^{-1-\alpha}  [\grad\Gamma(x) - \grad\Gamma(y)]\right) \theta(y)  \dd{y}  \label{eqn-523} \\ 
 &\hspace{-2.1em} + (\nabla_x^2  - \nabla_y^2)\left( |x-y|^{-1-\alpha}  [\grad\Gamma(x) - \grad\Gamma(y)]\right)  
 \\
 &\hspace{-3em}= \int_{\del\md} \nabla_y \left( |x-y|^{-1-\alpha}  [\grad\Gamma(x) - \grad\Gamma(y)]\right)   \theta(y)\nrml_{\out,\del \md}  \dd{l(y)} \label{eqn-524}  
 \\
 & \hspace{-2.1em}+\int_{\md }  \nabla_y(|x-y|^{-1-\alpha}  [\grad\Gamma(x) - \grad\Gamma(y)] )\nabla \theta(y)  \dd{y}+ O(|x-y|^{-1-\alpha}). \label{eqn-525}
\end{align}
The boundary term
 \eqref{eqn-524} from integation by parts is
\begin{align*}
  &\int_{\del\md} \nabla_y \left( |x-y|^{-1-\alpha}  [\grad\Gamma(x) - \grad\Gamma(y)]\right)   \theta(y) \nrml_{\out,\del\md} \dd{s}  
  \\
&  =  \frac{1}{2}\sum_{\sigma = \pm 1} \int_{\mathbb T} \nabla_y \left( |x-y|^{-1-\alpha}  [\grad\Gamma(x) - \grad\Gamma(y)]\right)\Big|_{y=\myspine^\sigma}   (\del_s \myspine^\sigma)^\perp \dd{s}.
\end{align*}
For the remaining term in \eqref{eqn-525}, we have \begin{align*}
&|\nabla_y(|x-y|^{-1-\alpha}  [\grad\Gamma(x) - \grad\Gamma(y)] ) |
 \\
 &=\left|\frac{(-1-\alpha)(x-y)(\nabla \Gamma(x) - \nabla \Gamma (y)) + \nabla^2 \Gamma(y)|x-y|^2}{|x-y|^{3+\alpha}}     \right|
 \\
 &= \frac{(1+\alpha)|\nabla^2\Gamma(x)|}{|x-y|^{1+\alpha}} + O(|x-y|^{-\alpha}).
\end{align*}
All of these terms \eqref{eqn-523}, \eqref{eqn-524}, \eqref{eqn-525} are $O(\delta^{-\alpha})$ terms, which can be seen by using the asymptotic lemma \ref{prop-asymptotic-integral} to compute the $s$ integral to leading order. For instance, for \eqref{eqn-525},
we write out the integral explicitly using the coordinates $y = \mycurve(s_*) + \delta \xi_*\nrml(s_*) $, $x = \mycurve(s) + \delta \xi \nrml(s)$,and the growth condition $|\nabla \theta|\lesssim  \frac1\delta$ from \eqref{eqn-growth-condition},
\begin{align}
    &\left|\int_{\md }  \nabla_y(|x-y|^{-1-\alpha}  [\grad\Gamma(x) - \grad\Gamma(y)] )\nabla \theta(y)  \dd{y}\right|\\
    & \le \frac{1 }{\delta} (1+\alpha) |\nabla^2 \Gamma(x)| \int_{\xi=f-2}^{f+2}  \int_{\mathbb T} \frac{\delta L(1-\delta\kappa_* \xi_*) } {|\mycurve-\mycurve_* + \delta (\xi \nrml-\xi_* \nrml_*)|^{1+\alpha} } \dd{s_*} \dd{\xi_*} \\ 
    &= O(\delta^{-\alpha}),
\end{align}
which follows by applying the asymptotic Lemma \ref{prop-asymptotic-integral}. Thus,  $\|\nabla^2Q\|_{L^\infty} = O(\delta^{-\alpha})$, and Proposition \ref{spine-property} implies that
\begin{align} 2B_2(t) &= \int_{\mathbb T} 
\br{ 
        \int_{\md }  |\myspine-y|^{-1-\alpha}  [\grad\Gamma(\myspine) - \grad\Gamma(y)]%
        \theta(y)   \dd{y}%
    }%
    \cdot \del_s\myspine \dd{s}\\ &+ O(\delta^{2-\alpha}).
\end{align}
Reversing the order of integration and using the tubular coordinates we notice we can again use the spine condition \eqref{eqn-spine-cancellation} to bring out an extra cancellation. That is, defining $H$ by
\[ H(y_*):=\int_{\mathbb T} |\myspine-y_*|^{-1-\alpha}  [\grad\Gamma(\myspine) - \grad\Gamma(y_*)]\cdot \del_s\myspine  \dd{s},\]
we deduce that
\begin{align*} 
2B_2(t) 
    &= \int_{f - 2}^{f+2}\int_{\mathbb T}%
        H(y_*) \Omega_*%
     \delta L_{1*} \dd{s}_* \dd{\xi}_* + O(\delta^{2-\alpha}) \\%
    &= \delta \int_{f - 2}^{f+2}\int_{\mathbb T} \Omega_* H(y_*)  L \dd{s}_* \dd{\xi}_*  + O(\delta^{2-\alpha})\\
    &= \delta \int_{f - 2}^{f+2}\int_{\mathbb T} \Omega_* \br{H(y_*) - H(\myspine_*)} L\dd{s}_* \dd{\xi}_*  + O(\delta^{2-\alpha})\\%
    &= O(\delta^{2-\alpha}).%
\end{align*}

The final term to estimate is
\[ B_0(t) = \int_{(\md)^c} \int_{(\md)^c} \grad_x^\perp |x-y|^{-1-\alpha} \cdot\grad\Gamma(x)\theta(x)\theta(y)  \dd{x}\dd{y}. \]
We use both gradients appearing in $B_0$ to integrate by parts (via the formulas \eqref{eqn-signsgrad},\eqref{eqn-signsgradperp}) , on which we obtain only boundary terms due to either of the two cancellations $\grad^\perp\cdot \grad = 0$ or $\grad\theta |_{\del \md} = 0$:
\begin{align*}
&B_0(t)\\
&= \int_{(\md)^c} \int_{(\md)^c} \grad_x^\perp |x-y|^{-1-\alpha} \cdot\grad\Gamma(x)\theta(x)\theta(y)  \dd{x}\dd{y}\\ 
&= -\frac12\sum_{\sigma_1=\pm 1}\int_{(\md)^c}\int_{\mathbb T} \grad_x^\perp |x-y|^{-1-\alpha}|_{x=\myspine^{\sigma_1}}\cdot (\del_s \myspine^{\sigma_1})^\perp \Gamma(\myspine^{\sigma_1}) \theta(y) \dd{s} \dd{y}\\
&= \frac12\sum_{\sigma_1=\pm 1}\int_{(\md)^c}\int_{\mathbb T} \grad_y^\perp |\myspine^{\sigma_1}-y|^{-1-\alpha}\cdot (\del_s \myspine^{\sigma_1})^\perp \Gamma(\myspine^{\sigma_1}) \theta(y) \dd{s} \dd{y}\\
&= \frac14\sum_{\sigma_1=\pm 1}\sum_{\sigma_2=\pm 1}\int_{\mathbb T}\int_{\mathbb T}  |\myspine^{\sigma_1}-\myspine^{\sigma_2}_*|^{-1-\alpha}\cdot (\del_s \myspine^{\sigma_1})^\perp\cdot \del_s\myspine^{\sigma_2}_* \Gamma(\myspine^{\sigma_1}) \dd{s} \dd{s_*}\\
&= -\frac14 \sum_{\sigma_1 = \pm 1}\sum_{\sigma_2 = \pm 1} \int_{\mathbb T} \int_{\mathbb T} |\myspine^{\sigma_1} - \myspine^{\sigma_2}_*|^{-1-\alpha} \del_s\myspine^{\sigma_1}  \cdot [\del_s\myspine^{\sigma_2}_*]^\perp \Gamma(\myspine^{\sigma_1})   \dd{s}\dd{s}_*.
\end{align*}

This sum of four terms will now be regrouped into two terms $B_0(t) = B_{00}(t) + B_{01}(t)$, one where $\sigma_1 = \sigma_2$ and one where $\sigma_1 = -\sigma_2$,
\begin{align*} 
B_{00}(t) 
&= -\frac14 \sum_{\sigma= \pm 1}   \int_{\mathbb T} \int_{\mathbb T} |\myspine^{\sigma} - [\myspine^{\sigma}]_*|^{-1-\alpha} \del_s\myspine^{\sigma} \cdot [\del_s\myspine^{\sigma}]^\perp_* \Gamma(\myspine^{\sigma})   \dd{s}\dd{s}_*, \\
B_{01}(t) 
&= -\frac14 \sum_{\sigma= \pm 1}   \int_{\mathbb T} \int_{\mathbb T} |\myspine^{\sigma} - [\myspine^{-\sigma}]_*|^{-1-\alpha} \del_s\myspine^{\sigma} \cdot [\del_s\myspine^{-\sigma}]^\perp_* \Gamma(\myspine^{\sigma})   \dd{s}\dd{s}_* .
\end{align*}
For $B_{00}(t)$, we can use the formula $f(a+b) + f(a-b) = 2f(a) + O(b^2)$ to get that
\[ B_{00}(t) = -\frac12 \int_{\mathbb T}\int_{\mathbb T} \frac{\del_s\myspine \cdot \del_s\myspine_*^\perp}{|\myspine - \myspine_*|^{1+\alpha}} \Gamma(\myspine)  \dd{s}\dd{s}_* + O(\delta^2).\] 
\newcommand{\Bdelta}{\mathcal B}
Treating $B_{01}(t)=:\Bdelta(\delta)$  as a function of $\delta$, we need to prove that
\[ \Bdelta(\delta) = \Bdelta(0) + O(\delta^{2-\alpha}),\]
since
$\Bdelta(0) = - \frac12 \int_{\mathbb T}\int_{\mathbb T} \frac{\del_s\myspine \cdot \del_s\myspine_*^\perp }{|\myspine - \myspine_*|^{1+\alpha}} \Gamma(\myspine)  \dd{s}\dd{s_*}$ with $B_{00}(t)$ gives the required evolution term:
\[ B_{00}(t) + \mathcal B(0) = \int_{\mathbb T}\left(\int_{\mathbb T} \frac{\del_s\myspine_* - \del_s\myspine }{|\myspine - \myspine_*|^{1+\alpha}}  \dd{s_*}\right)\cdot \del_s\myspine^\perp  \Gamma(\myspine) \dd{s}  + O(\delta^2).\]
Symmetrizing as before, we obtain
\[ \Bdelta(\delta) = -\frac{1}{8}\sum_{\sigma= \pm 1}   \int_{\mathbb T} \int_{\mathbb T} \frac{\del_s\myspine^{\sigma} \cdot [\del_s\myspine^{-\sigma}]^\perp_* }{|\myspine^{\sigma} - [\myspine^{-\sigma}]_*|^{1+\alpha} } \br{\Gamma(\myspine^{\sigma}) - \Gamma([\myspine^{-\sigma}]_*)}   \dd{s}\dd{s}_*. \]
Since  \eqref{eqn-spine-evolution} involves test functions supported on the curve $\myspine$, we may assume that $\Gamma$ does not depend on $\xi$ on a $\delta$ neighbourhood of $\myspine$, giving 
\[ \Bdelta(\delta) = -\frac{1}{8}\sum_{\sigma= \pm 1}   \int_{\mathbb T} \int_{\mathbb T} \frac{ \del_s\myspine^{\sigma} \cdot [\del_s\myspine^{-\sigma}]^\perp_*}{|\myspine^{\sigma} - [\myspine^{-\sigma}]_*|^{1+\alpha} } \br{\Gamma(\myspine) - \Gamma(\myspine_*)}   \dd{s}\dd{s}_*. \]
Note that $\Bdelta(0)$ has the well-behaved $O(|s-s_*|^{1-\alpha})$ integrand. Recalling that $\Bdelta'(\tilde\delta) = \frac{\Bdelta(\delta)-\Bdelta(0)}{\delta}$ for some $\tilde\delta\in(0,\delta)$, it suffices to prove that
\[ \Bdelta'(\delta) \overset{?}{=} O(\delta^{1-\alpha}),\]
since $s\mapsto s^{1-\alpha}$ is increasing for $0<s$. On differentiation with respect to $\delta$, a factor of $\sigma$ appears, which means the sum over $\sigma=\pm 1$ becomes a symmetric difference. We expand the shorthand notation $\del_s\myspine^{\sigma}\cdot [\del_s\myspine^{-\sigma}]^\perp_*$ to find the derivative in $\delta$,
\[
    \del_s\myspine^{\sigma}\cdot [\del_s\myspine^{-\sigma}]^\perp_* = \del_s\myspine \cdot \del_s\myspine^\perp_* + \sigma \delta(\del_s\myspine \cdot \tgt_* + \tgt \cdot \del_s\myspine_*) + O(\delta^2) .
\]
Hence, its $\delta$-derivative is some bounded function, say $E(s,s_*)$.
When $\del_\delta$ hits the kernel $|\myspine^{\sigma} - [\myspine^{-\sigma}]_*|^{-1-\alpha}$, we have
\begin{align*} & \del_\delta |\myspine^{\sigma} - [\myspine^{-\sigma}]_*|^{-1-\alpha} \\
&= (-1-\alpha)|\myspine^{\sigma} - [\myspine^{-\sigma}]_*|^{-3-\alpha}(\myspine^{\sigma} - [\myspine^{-\sigma}]_*)\cdot \del_\delta ( \myspine^{\sigma} - [\myspine^{-\sigma}]_*) \\    
&= (-1-\alpha)|\myspine^{\sigma} - [\myspine^{-\sigma}]_*|^{-3-\alpha}[(\myspine - \myspine_*)\cdot \sigma (\nrml+\nrml_*) + O(\delta)].   
\end{align*}
With the cancellation from the symmetrisation in $\Gamma$, we see that we have
\begin{align*}
&-8\Bdelta'(\delta)  = \sum_{\sigma=\pm1} \sigma \int_{\mathbb T} \int_{\mathbb T} |\myspine^\sigma - [\myspine^{-\sigma}]_*|^{-1-\alpha} (\Gamma(\myspine) - \Gamma(\myspine_*)) \times \\ & \quad  \Bigg( E(s,s_*)  + (-1-\alpha) \del_s\myspine\cdot[\del_s\myspine]^\perp_* \frac{\myspine^{\sigma} - [\myspine^{-\sigma}]_*}{|\myspine^{\sigma} - [\myspine^{-\sigma}]_*|^2} \Bigg)\dd{s}\dd{s}_* + O(\delta) .
\end{align*}
The factor of $\sigma$ means that we can use the Mean Value Theorem in the form $f(x+\delta) - f(x-\delta) =  O(\delta)$ to obtain that actually $\Bdelta'(\delta) = O(\delta)$, which finally implies the result. 
\end{proof}

%% file: asymptoticintegral.tex
\section{Asymptotic for a parameterised integral}
\label{sectionAppendix}
\begin{lem}\label{prop-asymptotic-integral} Let $\alpha\in(0,1)$ and $\Torus:=\Real / \mathbb Z$. For $\tau>0$, let $I=I(\tau)$ denote the following family of integrals,
\[ I = \int_{s\in \Torus} \frac{a(s)}{|g(s) + \tau^2|^{(1+\alpha)/2}} \ \dd{s}, \]
where $a=a(s),g=g(s) \in C^\infty (\Torus)$ and $g$ has 0 as its unique global minimum at $s=0$ that is non-degenerate, i.e. 
$g''(0)>0,\quad \operatorname{argmin} g = 0, \quad g(0) = \min g = 0.$
Then we have the asymptotic relation as $\tau\to0$,
\begin{align*}
    I 
    &= \frac{a(0)}{G} C_\alpha \tau^{-\alpha} 
    - \frac{a(0)(\alpha^{-1} 2^{1+\alpha} + b_\alpha)}{ G^{1+\alpha} }\\ 
    &+ \int_{s\in\Torus} \frac{a(s)}{|g(s)|^{(1+\alpha)/2}} - \frac{a(0)}{G^{1+\alpha}|\mysin s|^{1+\alpha}} \ \dd{s}  
    + O(\tau^{2-\alpha}),
    \quad  (\tau\to 0) ,
\end{align*} where:
\begin{enumerate}
    \item $\mysin s := \sin(\pi s)/\pi $,
    \item $C_\alpha$ is the constant $C_\alpha =  \frac{\sqrt{\pi}\Gamma(\frac{\alpha}{2})}{\Gamma(\frac{\alpha+ 1}{2})}<\infty$ (which diverges as $\alpha\to0$),
    \item $b_\alpha$ is the constant $b_\alpha:= \int_{-1/2}^{1/2} (\frac{1}{|s|^{1+\alpha}} - \frac{1}{|\mysin s|^{1+\alpha}}) \ \dd{s}< \infty$,
    \item $G$ is the constant $G:=\sqrt{g''(0)/2}$ 
		, 
		and
    \item the $O(\tau^{2-\alpha})$ constant depends on $W^{3,\infty}$ norms of $a$ and $g$.
\end{enumerate}
\end{lem}
\begin{proof}
The proof for the case $\alpha=0$ can be found in \cite{fefferman2012almost}. We take $(-1/2,1/2)$ as a fundamental domain for  $\mathbb T$.  We split $I=I_{\mynear} + I_{\myfar}$ into an integral $I_{\mynear}$ around the minimiser of $g$ and $I_{\myfar}$  on the complement,
\begin{align} A_\mynear &:= (s_- , s_+), & I_\mynear &:= \int_{A_\mynear} \frac{a(s)}{|g(s) + \tau^2|^{(1+\alpha)/2}} \ \dd{s}, \label{eqn-Anear}\\
    A_\myfar &:= \Torus \setminus A_\mynear, & I_\myfar &:= \int_{A_\myfar} \frac{a(s)}{|g(s) + \tau^2|^{(1+\alpha)/2}} \ \dd{s}, \label{eqn-Afar}
\end{align}
where $s_\pm $ are chosen (depending on $g$) sufficiently close to 0 so that we can choose new coordinates $\sigma$ such that $g(s) = \sigma^2$, and that $g(s_-) = g(s_+) =: \sigma_0^2 \ll 1$. Also define $\bar{a}(\sigma)$ so that $\bar{a}(\sigma)\dd{\sigma} = a(s)\dd{s}$ in the integral. Thus        $ I_\mynear = \int_{-\sigma_0}^{\sigma_0} \frac{\bar{a}(\sigma)}{|\sigma^2+\tau^2|^{(1+\alpha)/2}}  \dd{\sigma} = I_{\mynear,1} + I_{\mynear,2} + I_{\mynear,3} ,$ where
\begin{align*}
        I_{\mynear,1} &= \bar{a}(0) \int_{-\sigma_0}^{\sigma_0} \frac{1}{|\sigma^2+\tau^2|^{(1+\alpha)/2}}  \dd{\sigma}, \quad
        I_{\mynear,2} =  \int_{-\sigma_0}^{\sigma_0} \frac{\bar{a}(\sigma) - \bar{a}(0)}{|\sigma|^{1+\alpha}}  \dd{\sigma} ,\\
        I_{\mynear,3} &=  \int_{-\sigma_0}^{\sigma_0} \big(\bar{a}(\sigma) - \bar{a}(0) - \bar{a}'(0) \sigma \big)\left(\frac{1}{|\sigma^2+\tau^2|^{(1+\alpha)/2}} - \frac{1}{|\sigma|^{1+\alpha}} \right)  \dd{\sigma} .
\end{align*} 
By the regularity of $\bar{a}$, these integrals are well-defined. $I_{\mynear,1} $ is the only term that appears for a constant function $\bar{a} \equiv \bar{a}(0)$. $I_{\mynear,2}$ is bounded independent of $\tau$. $I_{\mynear,3}$ can easily be seen to be $O(\tau^{2-\alpha})$ using 
\[ \abs{ \frac1{|\sigma^2+\tau^2|^{(1+\alpha)/2}} - \frac1{\sigma^{1+\alpha}} } \le  \begin{cases}
 2 \sigma^{-1-\alpha} & |\sigma|\le \tau, \\
\tau^2 \sigma^{-3-\alpha}  & |\sigma|>\tau .
 \end{cases} \]
The first bound follows from the triangle inequality and $\frac1{|\sigma^2+\tau^2|^{(1+\alpha)/2}} \le \frac1{|\sigma|^{1+\alpha}} $; the second bound follows from $\frac{1+\alpha}2<1$ and the Mean Value Theorem applied to $f(x) = x^{-(1+\alpha)/2}$, i.e. for some $\theta\in(0,1)$,
\begin{align}
 f(x+h) - f(x) = f'(x+\theta h)h = -\left(\frac{1+\alpha}2\right)\frac{h}{|x+\theta h|^{(3+\alpha)/2}}, \quad \label{eqn-MVT}   
\end{align}
 with $x=\sigma^2,h=\tau^2$, and $ |x+\theta h|^{-(3+\alpha)/2} \le |x|^{-(3+\alpha)/2}$. This implies
\[
|I_{\mynear,3}| 
\lesssim_{\|a''\|_{L^\infty}}\int_{-\tau}^\tau  2\sigma^{2-1-\alpha}   \ \dd{\sigma} + \int_{ \tau \le |\sigma| \le \sigma_0} \tau^2 \sigma^{2-3-\alpha} \ \dd{\sigma} = O(\tau^{2-\alpha}).
\]

We focus now on $I_{\mynear,1}$.  Define
\begin{equation}
\begin{split}
J:=\int_{-\sigma_0}^{\sigma_0}  \frac{1}{|\sigma^2 + \tau^2|^{(1+\alpha)/2}}  \dd{\sigma} 
    &= 2 \tau^{-\alpha} \int_{0}^{\sigma_0/\tau}  \frac{1}{|\sigma^2 + 1|^{(1+\alpha)/2}}  \dd{\sigma}.
\end{split}    
\end{equation}
In contrast with the $\alpha=0$ case, the integrand is in $L^1(\Real)$, so we can easily write down the following expression with an error term,
\begin{equation}
    \begin{split}
        J 
        &= 2\tau^{-\alpha} \br{ \int_0^\infty \frac{\dd{\sigma}}{|\sigma^2+1|^{(1+\alpha)/2} } - \int_{\sigma_0/\tau}^\infty \frac{\dd{\sigma}}{|\sigma^2+1|^{(1+\alpha)/2} }} \\ 
        &= 2\tau^{-\alpha} \br{ \int_0^\infty \frac{\dd{\sigma}}{|\sigma^2+1|^{(1+\alpha)/2} } - \int_{\sigma_0/\tau}^\infty \frac{\dd{\sigma}}{\sigma^{1+\alpha}} } \\
       & +
         \int_{\sigma_0/\tau}^\infty   \frac{1}{\sigma^{1+\alpha}} -\frac{1}{|\sigma^2+1|^{(1+\alpha)/2} } \ \dd{\sigma}  = C_\alpha \tau^{-\alpha} - \frac{2}{\alpha \sigma_0^\alpha} + \text{Rem},
    \end{split}
\end{equation} 
where $C_\alpha = \int_\Real \frac{\dd{\sigma}}{|\sigma^2+1|^{(1+\alpha)/2}} = \frac{\sqrt{\pi}\Gamma(\frac{\alpha}{2})}{\Gamma(\frac{\alpha+ 1}{2})}<\infty $ and the remainder term $\text{Rem}$ satisfies (using \eqref{eqn-MVT} with $x=\sigma^2$, $h=1$ )
\[|\text{Rem}| \leq  2\tau^{-\alpha} \int_{\sigma_0/\tau}^\infty \frac{\dd{\sigma}}{\sigma^{3+\alpha}} = 2\tau^{-\alpha}\frac{\tau^{2+\alpha}}{(2+\alpha)\sigma_0^{2+\alpha}} = O(\tau^2), \quad \tau\to 0. \]
Hence we have that $J = C_\alpha \tau^{-\alpha} - \frac{2}{\alpha \sigma_0^\alpha} + O(\tau^2) $ as $\tau\to0$.


Recall the transformation's defining equation $g(s) = \sigma^2$. We can write $g(s) = g''(0)s^2/2 + O(s^3)$ since $s=0$ is a global nondegenerate minimum  of $g$ with $g(0) = 0$. Hence for $\sigma>0$ (and therefore $s>0$),
$ \sigma = s\sqrt{g''(0)/2+O(s)} =s\sqrt{g''(0)/2} + O(s^2) $ as $s\to 0^+$, by the differentiability in $h$ of $\sqrt{g''(0)/2 + h}$. The case $\sigma<0$ is treated similarly, leading to
\[ \sigma = \sqrt{g(s)} = \sqrt{\frac{g''(0)}{2}} s + O(s^2), \quad s\to 0.\] and hence $\dv{\sigma}{s}(s) \xrightarrow[s\to 0]{} \sqrt{\frac{g''(0)}{2}}$. Below, we use $G:=\sqrt{\frac{g''(0)}{2}}$. Since $a(s) = \dv{\sigma}{s}(s) \bar{a}(\sigma)$ we have
$ \bar{a}(0)  = a(0)/G,$ 
which allows us to rewrite $I_{\mynear,1}$,
\begin{align}
I_{\mynear,1} =  \frac{a(0)}{G}C_\alpha \tau^{-\alpha}  - 2\frac{a(0)}{\alpha G\sigma_0^\alpha }+ O(\tau^2), \quad \tau \to 0.\label{eqn-Inear1final}
\end{align}
Let us now treat $I_{\mynear,2}$. Let $0<\sigma_1\ll \sigma_0$, and let $s_{1-}<0,\  s_{1+} > 0$ be the two unique numbers such that $g(s_{1\pm}) = \sigma_1^2 $. Since $\bar{a}(\sigma)\dd{\sigma} = a(s)\dd{s}$, it is clear that $\frac{\bar{a}(\sigma)\dd{\sigma}}{|\sigma|^{1+\alpha}} = \frac{a(s)\dd{s}}{g(s)^{(1+\alpha)/2}}	$. Hence, we only need to rewrite the other term of the difference $\frac{(\bar{a}(\sigma) - \bar{a}(0))\dd{\sigma}}{|\sigma|^{1+\alpha}} $, which is $\bar{a}(0) \int_{\sigma_1}^{\sigma_0} \frac{\dd{\sigma}}{\sigma^{1+\alpha}}$. We would like to replace the integral in $\sigma$ with an integral in $s$. Observe that as $0<\sigma_1<\sigma_0$ and $0<s_{1+} < s_+$,
\[  \int_{\sigma_1}^{\sigma_0} \frac{\dd{\sigma}}{\sigma^{1+\alpha}} = \frac1\alpha\left(\frac1{\sigma_1^\alpha} - \frac1{\sigma_0^\alpha} \right),\text{ and } \int_{s_{1+}}^{s_+} \frac{\dd{s}}{s^{1+\alpha}} = \frac1\alpha\left(\frac1{s_{1+}^\alpha} - \frac1{s_+^\alpha} \right).\]
Thus, we have the following equality for any  constant ${\tilde C}$, 
\begin{equation} \int_{\sigma_1}^{\sigma_0} \frac{\dd{\sigma}}{\sigma^{1+\alpha}} ={\tilde C} \int_{s_{1+}}^{s_+} \frac{\dd{s}}{s^{1+\alpha}} + 
    \frac{1}{\alpha} \Bigg( \underbrace{\frac{1}{\sigma_1^\alpha} - \frac{{\tilde C}}{s_{1+}^\alpha}}_{\star} + \frac{{\tilde C}}{s_+^\alpha} - \frac1{\sigma_0^\alpha}\Bigg).  \end{equation}
Treating $s$ as a function $s=s(\sigma)$, the Inverse Function Theorem gives the asymptotic $s =  G^{-1} \sigma + O(\sigma^2)$ for $\sigma \ll 1$. 
Setting ${\tilde C} = G^{- \alpha}$, as then the terms marked with a star $\star$ become error terms for $\sigma_1 \ll 1$,
\[ \star = \frac{1}{\sigma_1^\alpha} - \frac{1}{(Gs_{1+})^\alpha} = \frac1{\sigma_1^\alpha} - \frac{1}{(\sigma_1 + O(\sigma_1^2))^{\alpha}} \le \frac{2 O(\sigma_1^2)}{\alpha \sigma_1^{1+\alpha}} = O(\sigma_1^{1-\alpha}) , 
 \]
We can do a similar analysis for the integral $\int_{-\sigma_0}^{-\sigma_1} \frac{\dd{\sigma}}{|\sigma|^{1+\alpha}}$, yielding
\[\int_{-\sigma_0}^{-\sigma_1} \frac{\dd{\sigma}}{|\sigma|^{1+\alpha}} = \frac1{G^{\alpha}}\int_{s_-}^{s_{1-}} \frac{\dd{s}}{|s|^{1+\alpha}} + \frac1\alpha \Bigg(\underbrace{ \frac1{\sigma_1^\alpha} -\frac1{|Gs_{1-}|^\alpha}}_{=O(\sigma_1^{1-\alpha})} + \frac1{|Gs_{-}|^\alpha} - \frac1{\sigma_0^\alpha}  \Bigg) \]
which together yield (as $\int_{-\sigma_0}^{-\sigma_1} \frac{\dd{\sigma}}{|\sigma|^{1+\alpha}} = \int_{\sigma_1}^{\sigma_0} \frac{\dd{\sigma}}{|\sigma|^{1+\alpha}}$) 
\begin{align}
2\int_{\sigma_0}^{\sigma_1} \frac{\dd{\sigma}}{|\sigma|^{1+\alpha}} 
= &\int_{s_1+}^{s_+} \frac{\dd{s}}{G^{\alpha}|s|^{1+\alpha}} + \int_{s_-}^{s_{1-}} \frac{\dd{s}}{G^{\alpha}|s|^{1+\alpha}} \\ & + \frac1{\alpha}\left(\frac1{|Gs_{+}|^\alpha}  + \frac1{|Gs_{-}|^\alpha} - \frac2{|\sigma_0|^\alpha} \right) + O(\sigma_1^{1-\alpha}) .
\end{align}
 Hence, we rewrite $I_{\mynear,2}$ as follows,
\begin{align*}
    I_{\mynear,2} 
    &=  \int_{-\sigma_0}^{\sigma_0} \frac{\bar{a}(\sigma) - \bar{a}(0)}{|\sigma|^{1+\alpha}} \ \dd{\sigma} \\
    &=  \lim_{\sigma_1\to 0}\br{  \int_{\sigma_1}^{\sigma_0} \frac{\bar{a}(\sigma)}{|\sigma|^{1+\alpha}} \ \dd{\sigma} + \int_{-\sigma_0}^{-\sigma_1} \frac{\bar{a}(\sigma)}{|\sigma|^{1+\alpha}} \ \dd{\sigma}  - 2\int_{\sigma_1}^{\sigma_0} \frac{\bar{a}(0)}{|\sigma|^{1+\alpha}} \ \dd{\sigma} } \\
    &= \lim_{\sigma_1\to 0}
    \Bigg(  \int_{s_{1+}}^{s_+} \frac{a(s)}{|g(s)|^{1+\alpha}} - \frac{a(0)}{G^{1+\alpha}|s|^{1+\alpha}} \ \dd{s} \\
    &\quad +   \int_{s_{-}}^{s_{1-}} \frac{a(s)}{|g(s)|^{1+\alpha}} - \frac{a(0)}{G^{1+\alpha}|s|^{1+\alpha}} \ \dd{s}  + O(\sigma_{1}^{1-\alpha}) \Bigg) \\ 
    &\quad - \frac{a(0)}{\alpha G^{1+\alpha} s_+^\alpha} 
        - \frac{a(0)}{\alpha G^{1+\alpha} |s_-|^\alpha}  
        + 2\frac{a(0)}{\alpha G \sigma_0^\alpha}\\
    &= \int_{s_-}^{s_+} \frac{a(s)}{|g(s)|^{1+\alpha}} 
    - \frac{a(0)}{G^{1+\alpha}|s|^{1+\alpha}} \ \dd{s} \\
    &\quad - \frac{a(0)}{\alpha G^{1+\alpha} s_+^\alpha} 
        - \frac{a(0)}{\alpha G^{1+\alpha}|s_-|^\alpha}  
        + 2\frac{a(0)}{\alpha G \sigma_0^\alpha}. \refstepcounter{equation} + O(\sigma_1^{1-\alpha}).\tag{\theequation}\label{eqn-addnsub}
\end{align*}
The term  $2\frac{a(0)}{\alpha G \sigma_0^\alpha}$ here in \eqref{eqn-addnsub} exactly cancels with the term with $-2\frac{a(0)}{\alpha G \sigma_0^\alpha}$ in the equation \eqref{eqn-Inearfinal} for $I_{\mynear,1}$. We therefore can write $I_\mynear$ as follows,
\begin{align*}
    I_\mynear 
    &= \frac{a(0)}{G}C_\alpha \tau^{-\alpha} 
    + \int_{s_-}^{s_+} \frac{a(s)}{|g(s)|^{1+\alpha}} - \frac{a(0)}{G^{1+\alpha}|s|^{1+\alpha}} \ \dd{s} \\
    &\quad - \frac{a(0)}{\alpha G^{1+\alpha} s_+^\alpha} 
    - \frac{a(0)}{\alpha G^{1+\alpha}|s_-|^\alpha} + O(\tau^{2-\alpha}),\quad  \tau\to 0. \refstepcounter{equation}\tag{\theequation}\label{eqn-Inearfinal}
\end{align*}
To finish, we need to include $I_{\myfar}$. Recall from \eqref{eqn-Afar} that $A_{\myfar}=\mathbb T\setminus (s_-,s_+)$. Note that with $s_\pm$ fixed, $g(s)^{-(1+\alpha)/2}$ is $L^\infty_s(A_\myfar)$, and the following error estimate holds, since $\frac{a(s)}{|g(s)+\tau^2|^{(1+\alpha)/2}}$ is smooth in $\tau\ll 1$:
\[ \int_{A_\myfar} \frac{a(s)\dd{s}}{|g(s)+\tau^2|^{(1+\alpha)/2}} = \int_{A_\myfar} \frac{a(s)\dd{s}}{|g(s)|^{(1+\alpha)/2}} + O(\tau^2).\] Since 
$ \int_{A_\myfar} \frac{\dd{s}}{|s|^{1+\alpha}} = \br{\int_{-1/2}^{s_-} + \int_{s_+}^{1/2}  } \frac{\dd{s}}{|s|^{1+\alpha}} = \frac{-2^{1+\alpha}}{\alpha} + \frac{1}{\alpha s_+^\alpha } + \frac{1}{\alpha |s_-|^\alpha }$,
We have
\begin{align*} 
I_\myfar 
&= \int_{A_\myfar} \frac{a(s)}{|g(s)|^{(1+\alpha)/2}} - \frac{a(0)}{G^{1+\alpha}|s|^{1+\alpha}}\dd{s}  
+ \frac{a(0)}{G^{1+\alpha}}\int_{A_\myfar}  \frac{1}{|s|^{1+\alpha}}\dd{s} + O(\tau^2) \\
&= \int_{A_\myfar} \frac{a(s)}{|g(s)|^{(1+\alpha)/2}} - \frac{a(0)}{G^{1+\alpha}|s|^{1+\alpha}}\dd{s} \\
&\quad \frac{-2^{1+\alpha}a(0)}{\alpha G^{1+\alpha}} + \frac{a(0)}{\alpha G^{1+\alpha} s_+^\alpha} 
        + \frac{a(0)}{\alpha G^{1+\alpha}|s_-|^\alpha}  + O(\tau^2). \refstepcounter{equation}\tag{\theequation}\label{eqn-Ifar}
\end{align*}
The terms $\frac{a(0)}{\alpha G^{1+\alpha} s_+^\alpha} 
        + \frac{a(0)}{\alpha G^{1+\alpha}|s_-|^\alpha}$  in \eqref{eqn-Ifar} cancel  $-\frac{a(0)}{\alpha G^{1+\alpha} s_+^\alpha} 
        - \frac{a(0)}{\alpha G^{1+\alpha}|s_-|^\alpha}$  in \eqref{eqn-Inearfinal}, leaving an expression that does not depend on $s_\pm$. 
Therefore, 
\begin{align*}
    I 
    &= \frac{a(0)}{G} C_\alpha \tau^{-\alpha} 
    - \frac{2^{1+\alpha}a(0)}{\alpha G^{1+\alpha} } \\
    &+ \int_{-1/2}^{1/2} \frac{a(s)}{|g(s)|^{(1+\alpha)/2}} - \frac{a(0)}{G^{1+\alpha}|s|^{1+\alpha}} \ \dd{s}
    + O(\tau^{2-\alpha}),\quad  \tau\to 0.
\end{align*} 
Since $a,g$ are 1-periodic functions, we rewrite this with the constant $b_\alpha:= \int_{-1/2}^{1/2} \frac{1}{|s|^{1+\alpha}} - \frac{\pi^{1+\alpha}}{|\sin(\pi s)|^{1+\alpha}} \ \dd{s} \in\mathbb R$,
\begin{align*}
     I 
     &= \frac{a(0)}{G} C_\alpha \tau^{-\alpha} 
     - \frac{a(0)(\alpha^{-1} 2^{1+\alpha} + b_\alpha)}{ G^{1+\alpha} } \\
     &+ \int_{s\in\Torus} \frac{a(s)}{|g(s)|^{(1+\alpha)/2}} - \frac{\pi^{1+\alpha}a(0)}{G^{1+\alpha}|\sin(\pi s)|^{1+\alpha}} \ \dd{s}  
     + O(\tau^{2-\alpha}),\quad  \tau\to 0 ,
\end{align*}
as claimed.
\end{proof}

Concerning the integral in Lemma \ref{prop-asymptotic-integral}, we have the following result.
\begin{cor}\label{cor-to-asymptotic-lemma} Let $\Torus:=\Real / \mathbb Z$. For $\delta\in (-\delta_0,\delta_0)$ sufficiently small, let $H=H(\delta)$ denote the following family of integrals,
\[ H = \int_{s\in\Torus} \br{ \frac{a(s)}{|g(s,\delta)|^{(1+\alpha)/2}} - \frac{ a(0)}{G(\delta)^{1+\alpha}|\mysin s|^{1+\alpha}} } \ \dd{s},\]
where $a=a(s)\in C^\infty (\Torus)$, $,g=g(s,\delta) \in C^\infty(\Torus \times [0,\infty))$, and $g$ has a unique global minimum that is non-degenerate with $\del_s^2 g(0,\cdot) > c > 0$ for a constant $c$ independent of $\delta$, and
\begin{equation*}\operatorname{argmin} g(\cdot,\delta) = 0, \quad g(0,\delta) = \min g(\cdot,\delta) = 0
\end{equation*}
and $G(\delta) := \sqrt{\del_s^2 g(0,\delta)/2}$.
Then we have the first order Taylor expansion $H(\tau) = H(0) + H'(0)\delta + O(\delta^2)$  for $\tau\ll 1$, with
\[ H'(0) = \int_{s\in\Torus} \br{ \frac{a(s)\del_\delta g(s,0) (-1-\alpha)}{|g(s,0)|^{(3+\alpha)/2}} - \frac{ \del_{\delta}(G^{-1-\alpha})(0) a(0)}{|\mysin s|^{1+\alpha}} } \ \dd{s}.\]
\end{cor}